\documentclass[11pt]{article}

\usepackage{mathrsfs}
\usepackage{hyperref}
\usepackage{amsmath,amsthm,amssymb,cite}
\usepackage{diagbox}
\usepackage{epic,eepic}
\usepackage{graphicx}
\usepackage{soul,color}
\usepackage{cases}
\usepackage{verbatim}
\usepackage{enumerate}
\usepackage{multirow}
\usepackage{indentfirst}
\usepackage[caption=false]{subfig}
\usepackage{algorithm}
\usepackage{algcompatible}
\usepackage{tikz}
\usetikzlibrary{shapes,arrows,chains}
\makeatletter

\newcommand{\Rmnum}[1]{\expandafter\@slowromancap\romannumeral #1@}
\makeatother
\newtheorem{theorem}{Theorem}[section]

\newtheorem{lemma}{Lemma}[section]

\allowdisplaybreaks[4]

\title{{Low-Rank Matrix Approximations with \\ Flip-Flop Spectrum-Revealing QR Factorization}\thanks{This work was funded by the CSC (grant 201606310121).}}

\author{
Yuehua Feng\thanks{School of Mathematical Science, Xiamen University, China. E-mail: {\tt fyh1001@hotmail.com}.}
\and Jianwei Xiao\thanks{Department of Mathematics, University of California, Berkeley. E-mail: {\tt jwxiao@berkeley.edu}.}
\and Ming Gu\thanks{Department of Mathematics, University of California, Berkeley. E-mail: {\tt mgu@math.berkeley.edu}.}
}

\begin{document}

\maketitle
\begin{abstract}
We present Flip-Flop Spectrum-Revealing QR (Flip-Flop SRQR) factorization, a significantly faster and more reliable variant of the QLP factorization of Stewart, for low-rank matrix approximations. Flip-Flop SRQR uses SRQR factorization to initialize a partial column pivoted QR factorization and then compute a partial LQ factorization. As observed by Stewart in his original QLP work, Flip-Flop SRQR tracks the exact singular values with ``considerable fidelity''. We develop singular value lower bounds and residual error upper bounds for Flip-Flop SRQR factorization. In situations where singular values of the input matrix decay relatively quickly, the low-rank approximation computed by SRQR is guaranteed to be as accurate as truncated SVD. We also perform a complexity analysis to show that for the same accuracy, Flip-Flop SRQR is faster than randomized subspace iteration for approximating the SVD, the standard method used in Matlab tensor toolbox. We also compare Flip-Flop SRQR with alternatives on two applications, tensor approximation and nuclear norm minimization, to demonstrate its efficiency and effectiveness. 
\end{abstract}

\medskip
{\small

}

\section{Introduction}
The singular value decomposition (SVD) of a matrix $A \in \mathbb{R}^{m\times n}$ is the factorization of $A$ into the product of three matrices
$A = U \Sigma V^T$ where $U = \left(u_1,\cdots,u_m\right) \in \mathbb{R}^{m \times m}$ and $V = \left(v_1,\cdots,v_n\right)\in \mathbb{R}^{n \times n}$ are orthogonal singular vector matrices and $\Sigma \in \mathbb{R}^{m \times n}$ is a rectangular diagonal matrix with non-increasing non-negative singular values $\sigma_i\left(1 \le i \le \min\left(m,n\right)\right)$ on the diagonal. The SVD has become a critical analytic tool in large data analysis and machine learning \cite{alter2000singular,elden2007matrix,narwaria2012svd}.

Let $\mbox{Diag}\left(x\right)$ denote the diagonal matrix with vector $x \in \mathbb{R}^n$ on its diagonal. For any $1 \le k \le \min\left(m,n\right)$, 
the rank-$k$ truncated SVD of A is defined by 
$$A_k \stackrel{def}{=} \left(u_1,\cdots,u_k\right) \mbox{Diag}\left(\sigma_1,\cdots,\sigma_k\right) \left(v_1,\cdots,v_k\right)^T.$$ 
The rank-$k$ truncated SVD turns out to be the best rank-$k$ approximation to $A$, as explained by Theorem \ref{Thm: best rank-k approximation to A}.

\begin{theorem} \label{Thm: best rank-k approximation to A}
(Eckart--Young--Mirsky Theorem \cite{eckart1936approximation,golub2012matrix}).
\begin{align*}
\min_{rank\left(C\right) \le k} \|A-C\|_2 &= \|A-A_k\|_2 = \sigma_{k+1}, \\
\min_{rank\left(C\right) \le k} \|A-C\|_F &= \|A-A_k\|_F = \sqrt{\sum_{j=k+1}^{\min\left(m,n\right)} \sigma_j^2},
\end{align*}
where $\|\cdot \|_2$ and $\|\cdot \|_F$ denote the $l_2$ operator norm and the Frobenius norm respectively. 
\end{theorem}

However, due to the prohibitive costs in computing the rank-$k$ truncated SVD, in practical applications one typically computes a rank-$k$ approximate SVD which satisfies some tolerance requirements \cite{drineas2006fast,gu2015subspace,halko2011finding,larsen1998lanczos,stewart1999qlp}. Then rank-$k$ approximate SVD has been applied to many research areas including principal component analysis (PCA) \cite{jolliffe1986principal,rokhlin2009randomized}, web search models \cite{kleinberg1999authoritative}, information retrieval \cite{berry1995using,furnas1988information}, and face recognition \cite{muller2004singular,turk1991eigenfaces}. 

Among assorted SVD approximation algorithms, the pivoted QLP decomposition proposed by Stewart \cite{stewart1999qlp} is an effective and efficient one. The pivoted QLP decomposition is obtained by computing a QR factorization with column pivoting \cite{businger1965linear,golub1965numerical} on $A$ to get an upper triangular factor $R$ and then computing an LQ factorization on $R$ to get a lower triangular factor $L$. Stewart's key numerical observation is that the diagonal elements of $L$ track the singular values of $A$ with ``considerable fidelity'' no matter the matrix $A$. The pivoted QLP decomposition is extensively analyzed in Huckaby and Chan \cite{HuckabyChan2003,HuckabyChan2005}. More recently, Deursch and Gu developed a much more efficient variant of the pivoted QLP decomposition, TUXV, and demonstrated its remarkable quality as a low-rank approximation empirically without a rigid justification of TUXV's success theoretically \cite{duersch2017randomized}. 

In this paper, we present Flip-Flop SRQR, a slightly different variant of TUXV of Deursch and Gu \cite{duersch2017randomized}. Like TUXV, Flip-Flop SRQR performs most of its work in computing a partial QR factorization using truncated randomized QRCP (TRQRCP) and a partial LQ factorization. Unlike TUXV, however, Flip-Flop SRQR also performs additional computations to ensure a spectrum-revealing QR factorization (SRQR) \cite{xiao2017fast} before the partial LQ factorization. 

We demonstrate the remarkable theoretical quality of this variant as a low-rank approximation, and its highly competitiveness with state-of-the-art low-rank approximation methods in real world applications in both low-rank tensor compression \cite{de2000multilinear,kolda2009tensor,savas2007handwritten,vannieuwenhoven2012new} and nuclear norm minimization \cite{cai2010singular,lin2010augmented,ma2011fixed,oh2015fast,toh2010accelerated}. 

The rest of this paper is organized as follows: In Section \ref{Sec:Prelim} we introduce the TRQRCP algorithm, the spectrum-revealing QR factorization, low-rank tensor compression, and nuclear norm minimization. In Section \ref{Sec:FFQR}, we introduce Flip-Flop SRQR and analyze its computational costs and low-rank approximation properties. In Section \ref{Sec:Num}, we present numerical experimental results comparing Flip-Flop SRQR with state-of-the-art low-rank approximation methods. 

\section{Preliminaries and Background}\label{Sec:Prelim}
\subsection{Partial QRCP}

\begin{algorithm}
\caption{Partial QRCP}\label{Alg: Partial QRCP}
\begin{algorithmic}
\STATE $\textbf{Inputs:}$
\STATE Matrix $A \in \mathbb{R}^{m \times n}$. Target rank $k$.
\STATE $\textbf{Outputs:}$
\STATE Orthogonal matrix $Q \in \mathbb{R}^{m \times m}$.
\STATE Upper trapezoidal matrix $R \in \mathbb{R}^{m \times n}$. 
\STATE Permutation matrix $\Pi \in \mathbb{R}^{n \times n}$ such that $A \Pi = QR$.
\STATE \textbf{Algorithm:}
\STATE Initialize $\Pi = I_n$. Compute column norms $r_s = \|A\left(1:m,s\right)\|_2~\left(1\le s \le n\right)$.
\FOR{$j=1:k$}
\STATE Find $i = \arg \max\limits_{j\le s \le n} r_s$. Exchange $r_j$ and $r_i$ columns in $A$ and $\Pi$.
\STATE Form Householder reflection $Q_j$ from $A\left(j:m,j\right)$. 
\STATE Update trailing matrix $A\left(j:m,j:n\right) \leftarrow Q_j^T A\left(j:m,j:n\right)$.
\STATE Update $r_s = \|A\left(j+1:m,s\right)\|_2~\left(j+1 \le s \le n\right)$.
\ENDFOR
\STATE $Q = Q_1 Q_2 \cdots Q_k$ is the product of all reflections. $R = $ upper trapezoidal part of $A$.
\end{algorithmic}
\end{algorithm}
The QR factorization of a matrix $A \in \mathbb{R}^{m \times n}$ is $A = QR$ with orthogonal matrix $Q \in \mathbb{R}^{m \times m}$ and upper trapezoidal matrix $R \in \mathbb{R}^{m \times n}$, which can be computed by LAPACK \cite{anderson1999lapack} routine xGEQRF, where x stands for the matrix data type. The standard QR factorization is not suitable for some practical situations where either the matrix $A$ is rank deficient or only representative columns of $A$ are of interest. Usually the QR factorization with column pivoting (QRCP) is adequate for the aforementioned situations except a few rare examples such as the Kahan matrix \cite{golub2012matrix}. Given a matrix $A \in \mathbb{R}^{m \times n}$, the QRCP of matrix $A$ has the form
\begin{equation*}
A \Pi = QR,
\end{equation*}
where $\Pi \in \mathbb{R}^{n \times n}$ is a permutation matrix, $Q \in \mathbb{R}^{m \times m}$ is an orthogonal matrix, and $R \in \mathbb{R}^{m \times n}$ is an upper trapezoidal matrix. QRCP can be computed by LAPACK \cite{anderson1999lapack} routines xGEQPF and xGEQP3, where xGEQP3 is a more efficient blocked implementation of xGEQPF. For given target rank $k~\left(1 \le k \le \min\left(m,n\right)\right)$, the partial QRCP factorization has a $2 \times 2$ block form 
\begin{equation}
A \Pi = Q \begin{pmatrix}R_{11} & R_{12}\\ & R_{22}\end{pmatrix} = \begin{pmatrix}Q_1 & Q_2\end{pmatrix} \begin{pmatrix}R_{11} & R_{12}\\ & R_{22}\end{pmatrix}, \label{Eqn:RRQR}
\end{equation}
where $R_{11} \in \mathbb{R}^{k \times k}$ is upper triangular. The details of partial QRCP are covered in Algorithm \ref{Alg: Partial QRCP}. The partial QRCP computes an approximate column subspace of $A$ spanned by the leading $k$ columns in $A \Pi$, up to the error term in $R_{22}$. Equivalently, \eqref{Eqn:RRQR} yields a low rank approximation 
\begin{equation}\label{Eqn:RRQR2}
A \approx Q_1 \begin{pmatrix}R_{11} & R_{12}\end{pmatrix} \Pi^T, 
\end{equation}
with approximation quality closely related to the error term in $R_{22}$. 

\begin{algorithm}
\caption{Truncated Randomized QRCP (TRQRCP)}\label{Alg: TRQRCP}
\begin{algorithmic}
\STATE $\textbf{Inputs:}$
\STATE Matrix $A \in \mathbb{R}^{m \times n}$. Target rank $k$. Block size $b$. Oversampling size $p\ge 0$.
\STATE $\textbf{Outputs:}$
\STATE Orthogonal matrix $Q \in \mathbb{R}^{m \times m}$.
\STATE Upper trapezoidal matrix $R \in \mathbb{R}^{k \times n}$. 
\STATE Permutation matrix $\Pi \in \mathbb{R}^{n \times n}$ such that $A \Pi \approx Q\left(:,1:k\right)R$.
\STATE $\textbf{Algorithm:}$
\STATE Generate i.i.d. Gaussian random matrix $\Omega \in \mathcal{N}\left(0,1\right)^{\left(b+p\right) \times m}$.
\STATE Form the initial sample matrix $B = \Omega A$ and initialize $\Pi = I_n$.
\FOR{$j=1:b:k$}
\STATE $b = \min\left(k-j+1,b\right)$.
\STATE Do partial QRCP on $B\left(:,j:n\right)$ to obtain $b$ pivots.
\STATE Exchange corresponding columns in $A,~B,~\Pi$ and $W^T$.
\STATE Do QR on $A\left(j:m,j:j+b-1\right)$ using WY formula without updating the trailing matrix.
\STATE Update $B\left(:,j+b:n\right)$.
\ENDFOR
\STATE $Q = Q_1 Q_2 \cdots Q_{\lceil k/b \rceil}$. $R = $ upper trapezoidal part of the submatrix $A\left(1:k,1:n\right)$.
\end{algorithmic}
\end{algorithm}
The Randomized QRCP (RQRCP) algorithm \cite{duersch2017randomized,xiao2017fast} is a more efficient variant of Algorithm \ref{Alg: Partial QRCP}. RQRCP generates a Gaussian random matrix $\Omega \in \mathcal{N}\left(0,1\right)^{(b+p) \times m}$ with $b + p \ll m$,  where the entries of $\Omega$ are independently sampled from normal distribution, to compress $A$ into $B = \Omega A$ with much smaller row dimension. In practice, $b$ is the block size and $p$ is the oversampling size. RQRCP repeatedly runs partial QRCP on $B$ to obtain $b$ column pivots, applies them to the matrix $A$, and then computes QR without pivoting (QRNP) on $A$ and updates the remaining columns of $B$. RQRCP exits this process when it reaches the target rank $k$. QRCP and RQRCP choose pivots on $A$ and $B$ respectively. RQRCP is significantly faster than QRCP as $B$ has much smaller row dimension than $A$. It is shown in \cite{xiao2017fast} that RQRCP is as reliable as QRCP up to failure probabilities that decay exponentially with respect to the oversampling size $p$.

Since the trailing matrix of $A$ is usually not required for low-rank matrix approximations (see \eqref{Eqn:RRQR2}), the TRQRCP (truncated RQRCP) algorithm of \cite{duersch2017randomized} re-organizes the computations in RQRCP to directly compute the approximation \eqref{Eqn:RRQR2}
without explicitly computing the trailing matrix $R_{22}$. For more details, both RQRCP and TRQRCP are based on the $WY$ representation of the Householder transformations  \cite{bischof1987wy,quintana1998blas,schreiber1989storage}: 
\begin{equation*}
Q = Q_1 Q_2 \cdots Q_k = I-YTY^T,
\end{equation*}
where $T \in \mathbb{R}^{k \times k}$ is an upper triangular matrix and $Y \in \mathbb{R}^{m \times k}$ is a trapezoidal matrix consisting of $k$ consecutive Householder vectors. Let $W^T \stackrel{def}{=} T^T Y^T A$, then the trailing matrix update formula becomes $Q^TA = A - YW^T$. The main difference between RQRCP and TRQRCP is that while RQRCP computes the whole trailing matrix update, TRQRCP only computes the part of the update that affects the approximation \eqref{Eqn:RRQR2}. More discussions about RQRCP and TRQRCP can be found in \cite{duersch2017randomized}. The main steps of TRQRCP are briefly described in Algorithm \ref{Alg: TRQRCP}. 

With TRQRCP, the TUXV algorithm (Algorithm 7 in \cite{duersch2017randomized}) computes a low-rank approximation with the QLP factorization at a greatly accelerated speed, by computing a partial QR factorization with column pivoting, followed with a partial LQ factorization. 

\subsection{Spectrum-revealing QR Factorization}
\begin{algorithm}
\caption{TUXV Algorithm}\label{Alg: TUXV}
\begin{algorithmic}[b]
\STATE $\textbf{Inputs:}$
\STATE Matrix $A \in \mathbb{R}^{m \times n}$. Target rank $k$. Block size $b$. Oversampling size $p\ge 0$.
\STATE $\textbf{Outputs:}$
\STATE Column orthonormal matrices $U \in \mathbb{R}^{m\times k},~V \in \mathbb{R}^{n \times k}$, and upper triangular matrix $R \in \mathbb{R}^{k\times k}$ such that $A \approx U R V^T$.
\STATE $\textbf{Algorithm:}$
\STATE Do TRQRCP on $A$ to obtain $Q \in \mathbb{R}^{m\times k},~R \in \mathbb{R}^{k\times n}$, and $\Pi \in \mathbb{R}^{n \times n}$.
\STATE $R = R \Pi^T$ and do LQ factorization, i.e., $[V,R] = qr(R^T,0)$.
\STATE Compute $Z = A V$ and do QR factorization, i.e., $[U,R] = qr(Z,0)$. 
\end{algorithmic}
\end{algorithm}
Although both RQRCP and TRQRCP are very effective practical tools for low-rank matrix approximations, they are not known to provide reliable low-rank matrix approximations due to their underlying greediness in column norm based pivoting strategy. To solve this potential problem of column based QR factorization, Gu and Eisenstat \cite{gu1996efficient} proposed an efficient way to perform additional column interchanges to enhance the quality of the leading $k$ columns in $A \Pi$ as a basis for the approximate column subspace. More recently, a more efficient and effective method, spectrum-revealing QR factorization (SRQR), was introduced and analyzed in \cite{xiao2017fast} to compute the low-rank approximation \eqref{Eqn:RRQR2}. The concept of spectrum-revealing, first introduced in \cite{xiao2016spectrum}, emphasizes the utilization of partial QR factorization  \eqref{Eqn:RRQR2} as a low-rank matrix approximation, as opposed to the more traditional rank-revealing factorization, which emphasizes the utility of the partial QR factorization  \eqref{Eqn:RRQR} as a tool for numerical rank determination. SRQR algorithm is described in Algorithm \ref{Alg: SRQR}. SRQR initializes a partial QR factorization using RQRCP or TRQRCP and then verifies an SRQR condition. If the SRQR condition fails, it will perform a pair-wise swap between a pair of leading column (one of first $k$ columns of $A \Pi$) and trailing column (one of the remaining columns). The SRQR algorithm will always run to completion with a high-quality 
low-rank matrix approximation \eqref{Eqn:RRQR2}. For real data matrices that usually have fast decaying singular-value spectrum, this approximation is often as good as the truncated SVD. The SRQR algorithm of \cite{xiao2017fast} explicitly updates the partial QR factorization \eqref{Eqn:RRQR} while swapping columns, but the SRQR algorithm can actually avoid any explicit computations on the trailing matrix $R_{22}$ using TRQRCP instead of RQRCP to obtain exactly the same partial QR initialization. Below we outline the SRQR algorithm. 

In \eqref{Eqn:RRQR}, let
\begin{equation}\label{Eqn:Rtilde}
\widetilde{R} \stackrel{def}{=} \begin{pmatrix}
R_{11} & a \\ & \alpha
\end{pmatrix}
\end{equation}
be the leading $\left(l+1\right) \times \left(l+1\right)$ submatrix of $R$. We define
\begin{equation}\label{Eqn:Rtilde2}
g_1 \stackrel{def}{=} \frac{\|R_{22}\|_{1,2}}{|\alpha|} \qquad \mbox{and} \qquad g_2 \stackrel{def}{=} |\alpha| \left\|\widetilde{R}^{-T}\right\|_{1,2},
\end{equation}
where $\|X\|_{1,2}$ is the largest column $2$-norm of $X$ for any given $X$. In \cite{xiao2017fast}, the authors proved approximation quality bounds involving $g_1, g_2$ for the low-rank approximation computed by RQRCP or TRQRCP. RQRCP or TRQRCP will provide a good low-rank matrix approximation if $g_1$ and $g_2$ are $O(1)$. The authors also proved that $g_1 \le \sqrt{\frac{1+\varepsilon}{1-\varepsilon}}$ and $g_2 \le \frac{\sqrt{2(1+\varepsilon)}}{1-\varepsilon} \left(1+\sqrt{\frac{1+\varepsilon}{1-\varepsilon}}\right)^{l-1}$ for RQRCP or TRQRCP, where $0 < \varepsilon < 1$ is a user-defined parameter which guides the choice of the oversampling size $p$. For reasonably chosen $\varepsilon$ like $\varepsilon = \frac{1}{2}$, $g_1$ is a small constant while $g_2$ can potentially be a extremely large number, which can lead to poor low-rank approximation quality. To avoid the potential exponential explosion of $g_2$, the SRQR algorithm (Algorithm \ref{Alg: SRQR}) proposed in \cite{xiao2017fast} uses a pair-wise swapping strategy to guarantee that $g_2$ is below some user defined tolerance $g > 1$ which is usually chosen to be a small number greater than one, like $2.0$.

\begin{algorithm}
\caption{Spectrum-revealing QR Factorization (SRQR)}\label{Alg: SRQR}
\begin{algorithmic}
\STATE $\textbf{Inputs:}$
\STATE Matrix $A \in \mathbb{R}^{m \times n}$. Target rank $k$. Block size $b$. Oversampling size $p\ge 0$. 
\STATE Integer $l \ge k$. Tolerance $g>1$ for $g_2$.
\STATE $\textbf{Outputs:}$
\STATE Orthogonal matrix $Q \in \mathbb{R}^{m \times m}$ formed by the first $k$ reflectors.
\STATE Upper trapezoidal matrix $R \in \mathbb{R}^{k \times n}$. 
\STATE Permutation matrix $\Pi \in \mathbb{R}^{n \times n}$ such that $A \Pi \approx Q\left(:,1:k\right)R$.
\STATE $\textbf{Algorithm:\,\,}$ 
\STATE Compute $Q, R, \Pi$ with RQRCP or TRQRCP to $l$ steps.
\STATE Compute squared 2-norm of the columns of $B(:,l+1:n): \widehat{r}_i \; (l+1 \le i \le n)$, where $B$ is a random projection of $A$ computed by RQRCP or TRQRCP.
\STATE Approximate squared 2-norm of the columns of $A(l+1:m,l+1:n): r_i = \widehat{r}_i/(b+p) \; (l+1 \le i \le n)$.
\STATE $\imath ={\bf argmax}_{l+1 \le i \le n}\{r_i\}$. 
\STATE Swap $\imath$-th and $(l+1)$-st columns of $A, \Pi, r$.
\STATE One-step QR factorization of $A(l+1:m,l+1:n)$.
\STATE $\left|\alpha\right| = R_{l+1,l+1}$.
\STATE $r_i = r_i - A(l+1,i)^2 \; (l+2 \le i \le n)$. 
\STATE Generate a random matrix $\Omega \in\mathcal{N}(0,1)^{d\times(l+1)} \; (d \ll l)$. 
\STATE Compute $g_2 = \left|\alpha\right| \left\|\widetilde{R}^{-T}\right\|_{1,2} \approx \frac{\left|\alpha\right|}{\sqrt{d}} \left\|\Omega \widetilde{R}^{-T}\right\|_{1,2}$.
\STATE \WHILE{$g_2 > g$}
\STATE $\imath ={\bf argmax}_{1 \le i \le l+1}\{i\text{th column norm of } \Omega \widetilde{R}^{-T}\}$.
\STATE Swap $\imath$-th and $(l+1)$-st columns of $A$ and $\Pi$ in a Round Robin rotation.
\STATE Givens-rotate $R$ back into upper-trapezoidal form.
\STATE $r_{l+1} = R_{l+1,l+1}^2$, $r_i = r_i + A(l+1,i)^2 \; (l+2 \le i \le n)$. 
\STATE $\imath ={\bf argmax}_{l+1 \le i \le n}\{r_i\}$. 
\STATE Swap $\imath$-th and $(l+1)$-st columns of $A, \Pi, r$.
\STATE One-step QR factorization of $A(l+1:m,l+1:n)$.
\STATE $\left|\alpha\right| = R_{l+1,l+1}$.
\STATE $r_i = r_i - A(l+1,i)^2 \; (l+2 \le i \le n)$.
\STATE Generate a random matrix $\Omega \in\mathcal{N}(0,1)^{d\times(l+1)} (d \ll l)$ .
\STATE Compute $g_2 = \left|\alpha\right| \left\|\widetilde{R}^{-T}\right\|_{1,2} \approx \frac{\left|\alpha\right|}{\sqrt{d}} \left\|\Omega \widetilde{R}^{-T}\right\|_{1,2}$.
\ENDWHILE
\end{algorithmic}
\end{algorithm}

\subsection{Tensor Approximation}
In this section we review some basic notations and concepts involving tensors. A more detailed discussion of the properties and applications of tensors can be found in the review \cite{kolda2009tensor}. A tensor is a $d$-dimensional array of numbers denoted by script notation $\mathcal{X} \in \mathbb{R}^{I_1 \times \cdots \times I_d}$ with entries given by
\begin{equation*}
x_{j_1,\dots,j_d}, \quad 1 \le j_1 \le I_1, \dots, 1 \le j_d \le I_d.
\end{equation*}

We use the matrix $X_{\left(n\right)} \in \mathbb{R}^{I_n \times \left(\Pi_{j \neq n}I_j\right)}$ to denote the $n$th mode unfolding of the tensor $\mathcal{X}$. Since this tensor has $d$ dimensions, there are altogether $d$-possibilities for unfolding. The $n$-mode product of a tensor $\mathcal{X} \in \mathbb{R}^{I_1 \times \cdots \times I_d}$ with a matrix $U \in \mathbb{R}^{k \times I_n}$ results in a tensor $\mathcal{Y} \in \mathbb{R}^{I_1 \times \cdots \times I_{n-1} \times k \times I_{n+1} \times \cdots \times I_d}$ such that
\begin{equation*}
y_{j_1,\dots,j_{n-1},j,j_{n+1},\dots,j_d} = \left(\mathcal{X} \times_n U\right)_{j_1,\dots,j_{n-1},j,j_{n+1},\dots,j_d} = \sum_{j_n=1}^{I_n} x_{j_1,\dots,j_d}u_{j,j_n}. 
\end{equation*}

\noindent Alternatively it can be expressed conveniently in terms of unfolded tensors: 
\begin{equation*}
\mathcal{Y} = \mathcal{X} \times_n U \Leftrightarrow Y_{\left(n\right)} = UX_{\left(n\right)},.
\end{equation*}

Decompositions of higher-order tensors have applications in signal processing \cite{de1998matrix,sidiropoulos2000parallel,de2004dimensionality}, numerical linear algebra \cite{de2000multilinear,kolda2001orthogonal,zhang2001rank}, computer vision \cite{vasilescu2002multilinear,shashua2005non,wang2003facial}, etc. Two particular tensor decompositions can be considered as higher-order extensions of the matrix SVD: CANDECOMP/PARAFAC (CP) \cite{carroll1970analysis,harshman1970foundations} decomposes a tensor as a sum of rank-one tensors, and the Tucker decomposition \cite{tucker1966some} is a higher-order form of principal component analysis. Given the definitions of mode products and unfolding of tensors, we can define the higher-order SVD (HOSVD) algorithm for producing a rank $\left(k_1,\dots,k_d\right)$ approximation to the tensor based on the Tucker decomposition format. The HOSVD algorithm \cite{de2000multilinear,kolda2009tensor} returns a core tensor $\mathcal{G} \in \mathbb{R}^{k_1 \times \cdots \times k_d}$ and a set of unitary matrices $U_j \in \mathbb{R}^{I_j \times k_j}$ for $j = 1,\dots,d$ such that 
\begin{equation*}
\mathcal{X} \approx \mathcal{G} \times_1 U_1 \cdots \times_d U_d,
\end{equation*}
where the right-hand side is called a Tucker decomposition. However, a straightforward generalization to higher-order ($d \ge 3$) tensors of the matrix Eckart--Young--Mirsky Theorem is not possible \cite{de2000best}; in fact, the best low-rank approximation is an ill-posed problem \cite{de2008tensor}. The HOSVD algorithm is outlined in Algorithm \ref{Alg: HOSVD}.

\begin{algorithm}
\caption{HOSVD}\label{Alg: HOSVD}
\begin{algorithmic}
\STATE $\textbf{Inputs:}$
\STATE Tensor $\mathcal{X} \in \mathbb{R}^{I_1 \times \cdots \times I_d}$ and desired rank $\left(k_1,\dots,k_d\right)$.
\STATE $\textbf{Outputs:}$
\STATE Tucker decomposition $\left[\mathcal{G};U_1,\cdots,U_d\right]$.
\STATE $\textbf{Algorithm:}$
\FOR{$j=1:d$}
\STATE Compute $k_j$ left singular vectors $U_j \in \mathbb{R}^{I_j \times k_j}$ of unfolding $X_{\left(j\right)}$. 
\ENDFOR
\STATE Compute core tensor $\mathcal{G} \in \mathbb{R}^{k_1 \times \cdots \times k_d}$ as 
\STATE 
\begin{equation*}
\mathcal{G} \stackrel{def}{=} \mathcal{X} \times_1 U_1^T \times_2 \cdots \times_d U_d^T.
\end{equation*}
\end{algorithmic}
\end{algorithm}

Since HOSVD can be prohibitive for large-scale problems, there has been a lot of literature to improve the efficiency of HOSVD computations without a noticeable deterioration in quality. One strategy for truncating the HOSVD, sequentially truncated HOSVD (ST-HOSVD) algorithm, was proposed in \cite{andersson1998improving} and studied by \cite{vannieuwenhoven2012new}. As was shown by \cite{vannieuwenhoven2012new}, ST-HOSVD retains several of the favorable properties of HOSVD while significantly reducing the computational cost and memory consumption. The ST-HOSVD is outlined in Algorithm \ref{Alg: ST-HOSVD}.

\begin{algorithm}
\caption{ST-HOSVD}\label{Alg: ST-HOSVD}
\begin{algorithmic}
\STATE $\textbf{Inputs:}$
\STATE Tensor $\mathcal{X} \in \mathbb{R}^{I_1 \times \cdots \times I_d}$, desired rank $\left(k_1,\dots,k_d\right)$, and processing order $p=\left(p_1,\cdots,p_d\right)$.
\STATE $\textbf{Outputs:}$
\STATE Tucker decomposition $\left[\mathcal{G};U_1,\cdots,U_d\right]$.
\STATE $\textbf{Algorithm:}$
\STATE Define tensor $\mathcal{G} \leftarrow \mathcal{X}$.
\FOR{$j = 1:d$}
\STATE $r = p_j$.
\STATE Compute exact or approximate rank $k_r$ SVD of the tensor unfolding $G_{\left(r\right)} \approx \widehat{U}_r \widehat{\Sigma}_r \widehat{V}_r^T$.
\STATE $U_{r} \leftarrow \widehat{U}_r$.
\STATE Update $G_{\left(r\right)} \leftarrow \widehat{\Sigma}_r \widehat{V}_r^T$, i.e., applying $\widehat{U}_r^T$ to $\mathcal{G}$.
\ENDFOR
\end{algorithmic}
\end{algorithm}

Unlike HOSVD, where the number of entries in tensor unfolding $X_{\left(j\right)}$ remains the same after each loop, the number of entries in $G_{\left(p_j\right)}$ decreases as $j$ increases in ST-HOSVD. In ST-HOSVD, one key step is to compute exact or approximate rank-$k_r$ SVD of the tensor unfolding. Well known efficient ways to compute an exact low-rank SVD include Krylov subspace methods \cite{lehoucq1998arpack}. There are also efficient randomized algorithms to find an approximate low-rank SVD \cite{halko2011finding}. In Matlab tensorlab toolbox \cite{Tensorlab2016}, the most efficient method, MLSVD\_RSI, is essentially ST-HOSVD with randomized subspace iteration to find approximate SVD of tensor unfolding.

\subsection{Nuclear Norm Minimization}
Matrix rank minimization problem appears ubiquitously in many fields such as Euclidean embedding \cite{fazel2003log,linial1995geometry}, control \cite{fazel2002matrix,mesbahi1997rank,recht2010guaranteed}, collaborative filtering \cite{candes2009exact,rennie2005fast,srebro2005maximum}, system identification \cite{liu2013nuclear,liu2009interior}, etc. Matrix rank minimization problem has the following form:
\begin{equation*}
\min_{X \in \mathcal{C}} \;\; \mbox{rank}\left(X\right) 
\end{equation*}
where $X \in \mathbb{R}^{m \times n}$ is the decision variable, and $\mathcal{C}$ is a convex set. In general, this problem is NP-hard due to the combinatorial nature of the function $\mbox{rank}\left(\cdot\right)$. To obtain a convex and more computationally tractable problem, $\mbox{rank}\left(X\right)$ is replaced by its convex envelope. In \cite{fazel2002matrix}, authors proved that the nuclear norm $\|X\|_*$ is the convex envelope of $\mbox{rank}\left(X\right)$ on the set $\{X\in \mathbb{R}^{m\times n}:~\|X\|_2 \le 1\}$. The nuclear norm of a matrix $X \in \mathbb{R}^{m \times n}$ is defined as 
\begin{equation*}
\|X\|_* \stackrel{def}{=} \sum_{i = 1}^q \sigma_i \left(X\right),
\end{equation*}
where $q = \mbox{rank}\left(X\right)$ and $\sigma_i \left(X\right)$'s are the singular values of $X$.

In many applications, the regularized form of nuclear norm minimization problem is considered:
\begin{equation*}
\min_{X \in \mathbb{R}^{m \times n}} f\left(X\right) + \tau \|X\|_*
\end{equation*}
where $\tau > 0$ is a regularization parameter. The choice of function $f\left(\cdot\right)$ is situational: $f\left(X\right) = \|M-X\|_1$ in robust principal component analysis (robust PCA) \cite{candes2011robust}, $f\left(X\right) = \|\pi_\Omega\left(M\right) - \pi_\Omega\left(X\right) ||_F^2$ in matrix completion \cite{cai2010singular}, $f\left(X\right) = \frac{1}{2}\|AX-B\|_F^2$ in multi-class learning and multivariate regression \cite{ma2011fixed}, where $M$ is the measured data, $\|\cdot\|_1$ denotes the $l_1$ norm, and $\pi_\Omega\left(\cdot\right)$ is an orthogonal projection onto the span of matrices vanishing outside of $\Omega$ so that $\left[\pi_\Omega\left(X\right)\right]_{i,j} = X_{ij}$ if $\left(i,j\right) \in \Omega$ and zero otherwise.

Many researchers have devoted themselves to solving the above nuclear norm minimization problem and plenty of algorithms have been proposed, including, singular value thresholding (SVT) \cite{cai2010singular}, fixed point continuous (FPC) \cite{ma2011fixed}, accelerated proximal gradient (APG) \cite{toh2010accelerated}, augmented Lagrange multiplier (ALM) \cite{lin2010augmented}. The most expensive part of these algorithms is in the computation of the truncated SVD. Inexact augmented Lagrange multiplier (IALM) \cite{lin2010augmented} has been proved to be one of the most accurate and efficient among them. We now describe IALM for robust PCA and matrix completion problems.

Robust PCA problem can be formalized as a minimization problem of sum of nuclear norm and scaled matrix $l_1$-norm (sum of matrix entries in absolute value):
\begin{equation}\label{Eqn: nuclear norm minimization of RPCA problem}
\min \|X\|_* + \lambda \|E\|_1, \qquad \mbox{subject to} \qquad M = X + E,
\end{equation}
where $M$ is measured matrix, $X$ has low-rank, $E$ is a error matrix and sufficiently sparse, and $\lambda$ is a positive weighting parameter.
Algorithm \ref{Alg: ialm_rpca} describes the details of IALM method to solve robust PCA \cite{lin2010augmented} problem, where $\|\cdot\|_M$ denotes the maximum absolute value of the matrix entries, and $\mathcal{S}_\omega\left(x\right) = \mbox{sgn}\left(x\right) \cdot \max\left(|x| - \omega,0\right)$ is the soft shrinkage operator \cite{hale2008fixed} where $x \in \mathbb{R}^n$ and $\omega > 0$.

\begin{algorithm}
\caption{Robust PCA Using IALM}\label{Alg: ialm_rpca}
\begin{algorithmic}
\STATE $\textbf{Inputs:}$
\STATE Measured matrix $M \in \mathbb{R}^{m \times n}$, positive number $\lambda, ~\mu_0,~\overline{\mu}$, tolerance $tol$, $\rho > 1$.
\STATE $\textbf{Outputs:}$
\STATE Matrix pair $\left(X_k,E_k\right)$.
\STATE $\textbf{Algorithm:}$
\STATE $k = 0$; $J\left(M\right) = \max\left(\|M\|_2,\|M\|_M\right); ~Y_0 = M / J\left(M\right)$; $E_0 = 0$;
\WHILE{not converged}
\STATE $\mathbf{\left(U,\Sigma,V\right) = \mbox{\bf svd}\left(M-E_k+\mu_k^{-1}Y_k\right)}$;
\STATE $X_{k+1} = U \mathcal{S}_{\mu_k^{-1}}\left(\Sigma\right) V^T$;
\STATE $E_{k+1} = \mathcal{S}_{\lambda \mu_k^{-1}} \left(M-X_{k+1}+\mu_k^{-1}Y_k\right)$;
\STATE $Y_{k+1} = Y_k + \mu_k\left(M - X_{k+1} - E_{k+1}\right)$;
\STATE Update $\mu_{k+1} = \min\left(\rho \mu_k,\overline{\mu}\right)$;
\STATE $k = k+1$;
\IF{$\|M - X_{k}-E_{k}\|_F/\|M\|_F < tol$} 
\STATE Break;
\ENDIF
\ENDWHILE
\end{algorithmic}
\end{algorithm}

Matrix completion problem \cite{candes2009exact,lin2010augmented} can be written in the form: 
\begin{equation}\label{Eqn: nuclear norm minimization of matrix completion problem}
\min_{X \in \mathbb{R}^{m \times n}} \|X\|_* \qquad \mbox{subject to} \qquad X+E = M,\quad \pi_\Omega\left(E\right) = 0,
\end{equation}
where $\pi_\Omega: ~ \mathbb{R}^{m \times n} \rightarrow \mathbb{R}^{m \times n}$ is an orthogonal projection that keeps the entries in $\Omega$ unchanged and sets those outside $\Omega$ zeros. In \cite{lin2010augmented}, authors applied IALM method on the matrix completion problem. We describe this method in Algorithm \ref{Alg: ialm_mc}, where $\overline{\Omega}$ is the complement of $\Omega$.

\begin{algorithm}
\caption{Matrix Completion Using IALM}\label{Alg: ialm_mc}
\begin{algorithmic}
\STATE $\textbf{Inputs:}$
\STATE Sampled set $\Omega$, sampled entries $\pi_\Omega\left(M\right)$, positive number $\lambda, ~\mu_0,~\overline{\mu}$, tolerance $tol$, $\rho > 1$.
\STATE $\textbf{Outputs:}$
\STATE Matrix pair $\left(X_k,E_k\right)$.
\STATE $\textbf{Algorithm:}$
\STATE $k = 0$; $Y_0 = 0$; $E_0 = 0$;
\WHILE{not converged}
\STATE $\mathbf{\left(U,\Sigma,V\right) = \mbox{\bf svd}\left(M-E_k+\mu_k^{-1}Y_k\right)}$;
\STATE $X_{k+1} = U \mathcal{S}_{\mu_k^{-1}}\left(\Sigma\right) V^T$;
\STATE $E_{k+1} = \pi_{\overline{\Omega}}\left(M-X_{k+1}+\mu_k^{-1}Y_k\right)$;
\STATE $Y_{k+1} = Y_k + \mu_k\left(M - X_{k+1} - E_{k+1}\right)$;
\STATE Update $\mu_{k+1} = \min\left(\rho \mu_k,\overline{\mu}\right)$;
\STATE $k = k+1$;
\IF{$\|M - X_{k}-E_{k}\|_F/\|M\|_F < tol$} 
\STATE Break;
\ENDIF
\ENDWHILE
\end{algorithmic}
\end{algorithm}

\section{Flip-Flop SRQR Factorization}\label{Sec:FFQR}
\subsection{Flip-Flop SRQR Factorization}
In this section, we introduce our Flip-Flop SRQR factorization, a slightly different from TUXV (Algorithm \ref{Alg: TUXV}), to compute SVD approximation based on QLP factorization. Given integer $l\ge k$, we run SRQR (the version without computing the trailing matrix) to $l$ steps on $A$, 
\begin{equation}\label{Eqn: partial QRCP on A}
A \Pi = Q R = Q \left(
\begin{array}{cc}
R_{11} & R_{12} \\
& R_{22}
\end{array}
\right),
\end{equation}
where $R_{11} \in \mathbb{R}^{l \times l}$ is upper triangular; $R_{12} \in \mathbb{R}^{l \times \left(n-l\right)}$; and $R_{22} \in \mathbb{R}^{\left(m-l\right) \times \left(n-l\right)}$. Then we run partial QRNP to $l$ steps on $R^T$,
\begin{equation} \label{Eqn: partial QRNP on R^T}
R^T = \left(
\begin{array}{cc}
R_{11}^T & \\
R_{12}^T & R_{22}^T
\end{array}
\right) = \widehat{Q} \left(
\begin{array}{cc}
\widehat{R}_{11} & \widehat{R}_{12} \\
& \widehat{R}_{22}
\end{array}
\right) \approx \widehat{Q}_1 \left(
\begin{array}{cc}
\widehat{R}_{11} & \widehat{R}_{12}
\end{array}
\right),
\end{equation}
where $\widehat{Q} = \left(
\begin{array}{cc}
\widehat{Q}_1 & \widehat{Q}_2
\end{array}
\right)$ with $\widehat{Q}_1 \in \mathbb{R}^{n \times l}$. Therefore, combing the fact that $A \Pi \widehat{Q}_1 = Q \begin{pmatrix} \widehat{R}_{11} & \widehat{R}_{12} \end{pmatrix}^T$, we can approximate matrix $A$ by 
\begin{equation}\label{Eqn: approximate matrix A by A (Pi widehat(Q1)) (Pi widehat(Q1))^T}
A = QR\Pi^T = Q \left(R^T\right)^T \Pi^T \approx Q \left(
\begin{array}{c}
\widehat{R}_{11}^T \\
\widehat{R}_{12}^T
\end{array} 
\right) \widehat{Q}_1^T \Pi^T = A \left(\Pi \widehat{Q}_1\right) \left(\Pi \widehat{Q}_1\right)^T.
\end{equation}

We denote the rank-$k$ truncated SVD of $A\Pi \widehat{Q}_1$ by $\widetilde{U}_k \Sigma_k \widetilde{V}_k^T$. 
Let $U_k = \widetilde{U}_k,~V_k = \Pi \widehat{Q}_1 \widetilde{V}_k$, then using \eqref{Eqn: approximate matrix A by A (Pi widehat(Q1)) (Pi widehat(Q1))^T}, a rank-$k$ approximate SVD of $A$ is obtained:
\begin{equation} \label{Eqn: rank-k approximate svd of A by A (Pi widehat(Q1)) (Pi widehat(Q1))^T}
A \approx U_k \Sigma_k V_k^T,
\end{equation}
where $U_k \in \mathbb{R}^{m \times k},~V_k \in \mathbb{R}^{n \times k}$ are column orthonormal; and $\Sigma_k = \mbox{Diag}\left(\sigma_1,\cdots,\sigma_k\right)$ with $\sigma_i$'s are the leading $k$ singular values of $A \Pi \widehat{Q}_1$. The Flip-Flop SRQR factorization is outlined in Algorithm \ref{Alg:flip-Flop SRQR factorization}.

\begin{algorithm}
\caption{Flip-Flop Spectrum-Revealing QR Factorization}
\label{Alg:flip-Flop SRQR factorization}
\begin{algorithmic}
\STATE $\textbf{Inputs:}$
\STATE Matrix $A \in \mathbb{R}^{m \times n}$. Target rank $k$. Block size $b$. Oversampling size $p\ge 0$. 
\STATE Integer $l \ge k$. Tolerance $g>1$ for $g_2$.
\STATE $\textbf{Outputs:}$
\STATE $U\in \mathbb{R}^{m \times k}$ contains the approximate top $k$ left singular vectors of $A$.
\STATE $\Sigma \in \mathbb{R}^{k \times k}$ contains the approximate top $k$ singular values of $A$.
\STATE $V\in \mathbb{R}^{n \times k}$ contains the approximate top $k$ right singular vectors of $A$.
\STATE $\textbf{Algorithm:}$
\STATE Run SRQR on $A$ to $l$ steps to obtain $\left(R_{11},R_{12}\right)$.
\STATE Run QRNP on $\left(R_{11},R_{12}\right)^T$ to obtain $\widehat{Q}_1$, represented by a sequence of Householder vectors.
\STATE $tmp = A \Pi \widehat{Q}_1$.
\STATE $[U_{tmp},\Sigma_{tmp},V_{tmp}] = svd\left(tmp\right)$.
\STATE $U = U_{tmp}\left(:,1:,k\right), \Sigma = \Sigma_{tmp}\left(1:k,1:k\right), V = \Pi \widehat{Q}_1 V_{tmp}\left(:,1:k\right)$.
\end{algorithmic}
\end{algorithm}

\subsection{Complexity Analysis}
In this section, we do complexity analysis of Flip-Flop SRQR. Since approximate SVD only makes sense when target rank $k$ is small, we assume $k \le l  \ll \min\left(m,n\right)$. The complexity analysis of Flip-Flop SRQR is as follows:
\begin{enumerate}
    \item The cost of doing SRQR with TRQRCP on $A$ is $2mnl+2(b+p)mn+(m+n)l^2$.
    \item The cost of QR factorization on $\left(R_{11}, R_{12}\right)^T$ and forming $\widehat{Q}_1$ is $2n l^2-\frac{2}{3} l^3$. 
    \item The cost of computing $tmp = A \Pi \widehat{Q}_1$ is $2mn l$.
    \item The cost of computing $[U,\sim,\sim] = svd\left(tmp\right)$ is $O(m l^2)$.
    \item The cost of forming $V_k$ is $2n l k$.
\end{enumerate}

Since $k \le l \ll \min\left(m,n\right)$, the complexity of Flip-Flop SRQR is $4mnl+2(b+p)mn$ by omitting the lower-order terms.

On the other hand, the complexity of approximate SVD with randomized subspace iteration (RSISVD) \cite{gu2015subspace,halko2011finding} is $\left(4+4q\right)mn\left(k+p\right)$, where $p$ is the oversampling size and $q$ is the number of subspace iterations (see detailed analysis in the appendix). In practice $p$ is chosen to be a small integer like $5$ in RSISVD and $l$ is usually chosen to be a little bit larger than $k$, like $l = k + 5$. Therefore, we can see that Flip-Flop SRQR is more efficient than RSISVD for any $q > 0$.

\subsection{Quality Analysis of Flip-Flop SRQR}
This section is devoted to the quality analysis of Flip-Flop SRQR. We start with Lemma \ref{Lem: singular inequality between X and X1 and X2}.

\begin{lemma} \label{Lem: singular inequality between X and X1 and X2}
Given any matrix $X = (X_1, X_2)$ with $X_i \in \mathbb{R}^{m \times n_i}~\left(i = 1,2\right)$ and $n_1 + n_2 = n$,
\begin{equation*}
\sigma_j \left(X\right)^2 \le \sigma_j\left(X_1\right)^2 + \|X_2\|_2^2 \quad \left(1 \le j \le \min(m,n)\right).
\end{equation*}
\end{lemma}
\begin{proof}
Since $X X^T = X_1 X_1^T + X_2 X_2^T$, we obtain the above result using \cite[Theorem 3.3.16]{horn1991topics}.
\end{proof}

We are now ready to derive bounds on the singular values and approximation error of Flip-Flop SRQR. We need to emphasize that even if the target rank is $k$, we run Flip-Flop SRQR with an actual target rank $l$ which is a little bit larger than $k$. The difference between $k$ and $l$ can create a gap between singular values of $A$ so that we can obtain a reliable low-rank approximation. 

\begin{theorem}
Given matrix $A \in \mathbb{R}^{m \times n}$, target rank $k$, oversampling size $p$, and an actual target rank $l \ge k$, $U_k,~\Sigma_k,~V_k$ computed by \eqref{Eqn: rank-k approximate svd of A by A (Pi widehat(Q1)) (Pi widehat(Q1))^T} of Flip-Flop SRQR satisfies
\begin{equation}\label{Eqn: the auxiliary lower bound of sigma_j(Sigma_k)}
\sigma_j(\Sigma_k) \ge \frac{\sigma_j(A)}{\sqrt[4]{1 + \frac{2\left\|R_{22}\right\|_2^4}{\sigma_j^4(\Sigma_k)}}} \quad \left(1 \le j \le k\right),
\end{equation}
and 
\begin{equation}\label{Eqn: the auxiliary upper bound of approximate SVD}
\left\|A - U_k \Sigma_k V_k^T \right\|_2
\le
\sigma_{k+1}\left(A\right) \sqrt[4]{1 + 2 \left(\frac{\left\|R_{22}\right\|_2}{\sigma_{k+1}\left(A\right)}\right)^4},
\end{equation}
where $R_{22} \in \mathbb{R}^{(m-l) \times (n-l)}$ is the trailing matrix in \eqref{Eqn: partial QRCP on A}. Using the properties of SRQR, we can further have 
\begin{equation}\label{Eqn: the lower bound of sigma_j(Sigma_k)}
\sigma_j\left(\Sigma_k\right) \geq \frac{\sigma_j\left(A\right)}{\sqrt[4]{1+\min \left( 2\widehat{\tau}^4, \tau^4\left(2+4\widehat{\tau}^4\right) \left(\frac{\sigma_{l+1}\left(A\right)}{\sigma_j\left(A\right)}\right)^4\right)}}\quad \left(1 \le j \le k\right),
\end{equation}
and 
\begin{equation}\label{Eqn: the upper bound of approximate SVD}
\left\|A - U_k \Sigma_k V_k^T \right\|_2
\le
\sigma_{k+1}\left(A\right) \sqrt[4]{1 + 2 \tau^4 \left(\frac{\sigma_{l+1}\left(A\right)}{\sigma_{k+1}\left(A\right)}\right)^4},
\end{equation}
where $\tau$ and $\widehat{\tau}$ defined in \eqref{Eqn: definition of tau and widehat(tau)} have matrix dimensional dependent upper bounds:
\begin{equation*}
\tau \le g_1 g_2 \sqrt{\left(l+1\right)\left(n-l\right)}, \quad \mbox{and} \quad \widehat{\tau} \le g_1 g_2 \sqrt{l\left(n-l\right)},
\end{equation*}
where $g_1 \le \sqrt{\frac{1+\varepsilon}{1-\varepsilon}}$ and $g_2 \le g$. $\varepsilon > 0$ and $g > 1$ are user defined parameters.
\end{theorem}

\begin{proof}
In terms of the singular value bounds, observe that 
\begin{equation*}
\left(
\begin{array}{cc}
\widehat{R}_{11} & \widehat{R}_{12} \\
& \widehat{R}_{22}
\end{array}
\right)
\left(
\begin{array}{cc}
\widehat{R}_{11} & \widehat{R}_{12} \\
& \widehat{R}_{22}
\end{array}
\right)^T 
= 
\left(
\begin{array}{cc}
\widehat{R}_{11}\widehat{R}_{11}^T + \widehat{R}_{12}\widehat{R}_{12}^T & \widehat{R}_{12}\widehat{R}_{22}^T \\
\widehat{R}_{22}\widehat{R}_{12}^T & \widehat{R}_{22}\widehat{R}_{22}^T
\end{array}
\right),
\end{equation*}
we apply Lemma \ref{Lem: singular inequality between X and X1 and X2} twice for any $1 \le j \le k$,
\begin{align}\label{Eqn: singular value estimate}
& \sigma_j^2 \left(
\left(
\begin{array}{cc}
\widehat{R}_{11} & \widehat{R}_{12} \\
& \widehat{R}_{22}
\end{array}
\right)
\left(
\begin{array}{cc}
\widehat{R}_{11} & \widehat{R}_{12} \\
& \widehat{R}_{22}
\end{array}
\right)^T
\right) \notag \\
\le \, & \sigma_j^2\left(
\left(
\begin{array}{cc}
\widehat{R}_{11}\widehat{R}_{11}^T + \widehat{R}_{12}\widehat{R}_{12}^T & \widehat{R}_{12}\widehat{R}_{22}^T
\end{array}
\right)
\right)
+ 
\left\|\left(
\begin{array}{cc}
\widehat{R}_{22}\widehat{R}_{12}^T & \widehat{R}_{22}\widehat{R}_{22}^T
\end{array}
\right)
\right\|_2^2 \notag \\
\le \, & \sigma_j^2\left(
\widehat{R}_{11}\widehat{R}_{11}^T + \widehat{R}_{12}\widehat{R}_{12}^T
\right) + \left\|\widehat{R}_{12}\widehat{R}_{22}^T\right\|_2^2 + \left\|\left(
\begin{array}{cc}
\widehat{R}_{22}\widehat{R}_{12}^T & \widehat{R}_{22}\widehat{R}_{22}^T
\end{array}
\right)
\right\|_2^2 \notag \\
\le \, & \sigma_j^2\left(
\widehat{R}_{11}\widehat{R}_{11}^T + \widehat{R}_{12}\widehat{R}_{12}^T
\right) + 2 \left\|\left(
\begin{array}{c}
\widehat{R}_{12} \\
\widehat{R}_{22}
\end{array}
\right)\right\|_2^4 \notag \\
= \, & \sigma_j^2\left(
\widehat{R}_{11}\widehat{R}_{11}^T + \widehat{R}_{12}\widehat{R}_{12}^T
\right) + 2 \left\|R_{22}\right\|_2^4.
\end{align}

\noindent The relation \eqref{Eqn: singular value estimate} can be further rewritten as 
\begin{equation*}
\sigma_j^4\left(A\right) \le \sigma_j^4\left(\left(
\begin{array}{cc}
\widehat{R}_{11} & \widehat{R}_{12}
\end{array}
\right)\right) + 2 \left\|R_{22}\right\|_2^4 = \sigma_j^4\left(\Sigma_k\right) + 2 \left\|R_{22}\right\|_2^4 \quad \left(1 \le j \le k\right),
\end{equation*}
\noindent which is equivalent to 
\begin{equation*}
\sigma_j(\Sigma_k) \ge \frac{\sigma_j(A)}{\sqrt[4]{1 + \frac{2\left\|R_{22}\right\|_2^4}{\sigma_j^4(\Sigma_k)}}}.
\end{equation*}

For the residual matrix bound, we let 
\begin{equation*}
\left(
\begin{array}{cc}
\widehat{R}_{11} & \widehat{R}_{12}
\end{array}
\right)
\stackrel{def}{=}
\left(
\begin{array}{cc}
\overline{R}_{11} & \overline{R}_{12}
\end{array}
\right) + \left(
\begin{array}{cc}
\delta \overline{R}_{11} & \delta \overline{R}_{12}
\end{array}
\right),
\end{equation*}
where $\left(
\begin{array}{cc}
\overline{R}_{11} & \overline{R}_{12}
\end{array}
\right)$ is the rank-$k$ truncated SVD of $\left(
\begin{array}{cc}
\widehat{R}_{11} & \widehat{R}_{12}
\end{array}
\right)$. Notice that 
\begin{equation} \label{Eqn: equation of 2-norm error}
\left\|A-U_k \Sigma_k V_k^T \right\|_2 = \left\|A \Pi - Q \left(
\begin{array}{cc}
\overline{R}_{11} & \overline{R}_{12} \\
& 0
\end{array}
\right)^T \widehat{Q}^T\right\|_2,
\end{equation}
\noindent it follows from the orthogonality of singular vectors that 
\begin{equation*}
\left(
\begin{array}{cc}
\overline{R}_{11} & \overline{R}_{12}
\end{array}
\right)^T
\left(
\begin{array}{cc}
\delta \overline{R}_{11} & \delta \overline{R}_{12}
\end{array}
\right) = 0,
\end{equation*}
and therefore
\begin{eqnarray*}
&& \left(
\begin{array}{cc}
\widehat{R}_{11} & \widehat{R}_{12}
\end{array}
\right)^T
\left(
\begin{array}{cc}
\widehat{R}_{11} & \widehat{R}_{12}
\end{array}
\right) \\
&=& \left(
\begin{array}{cc}
\overline{R}_{11} & \overline{R}_{12}
\end{array}
\right)^T
\left(
\begin{array}{cc}
\overline{R}_{11} & \overline{R}_{12}
\end{array}
\right) + 
\left(
\begin{array}{cc}
\delta \overline{R}_{11} & \delta \overline{R}_{12}
\end{array}
\right)^T
\left(
\begin{array}{cc}
\delta \overline{R}_{11} & \delta \overline{R}_{12}
\end{array}
\right), 
\end{eqnarray*}
which implies 
\begin{equation}\label{Eqn: decompose widehat R12}
\widehat{R}_{12}^T\widehat{R}_{12} = \overline{R}_{12}^T \overline{R}_{12} + \left(\delta \overline{R}_{12}\right)^T \left(\delta \overline{R}_{12}\right).
\end{equation}

\noindent Similar to the deduction of \eqref{Eqn: singular value estimate}, from \eqref{Eqn: decompose widehat R12} we can derive 
\begin{eqnarray*}
\left\|\left(
\begin{array}{cc}
\delta \overline{R}_{11} & \delta \overline{R}_{12} \\
& \widehat{R}_{22}
\end{array}
\right)\right\|_2^4 &\le& \left\|\left(
\begin{array}{cc}
\delta \overline{R}_{11} & \delta \overline{R}_{12}
\end{array}
\right)\right\|_2^4 + 2\left\|\left(
\begin{array}{c}
\delta \overline{R}_{12} \\
\widehat{R}_{22}
\end{array}
\right)\right\|_2^4 \\
&\leq &
\sigma_{k+1}^4\left(A\right) + 2 \left\|\left(
\begin{array}{c}
\widehat{R}_{12} \\
\widehat{R}_{22}
\end{array}
\right)
\right\|_2^4 = \sigma_{k+1}^4\left(A\right) + 2 \left\|R_{22}\right\|_2^4.
\end{eqnarray*}

\noindent Combining with \eqref{Eqn: equation of 2-norm error}, it now follows that 
\begin{eqnarray*}
\left\|A - U_k \Sigma_k V_k^T \right\|_2 &=& 
\left\|A\Pi - Q \left(
\begin{array}{cc}
\overline{R}_{11} & \overline{R}_{12} \\
& 0
\end{array}
\right)^T \widehat{Q}^T\right\|_2
=\left\|\left(
\begin{array}{cc}
\delta \overline{R}_{11} & \delta \overline{R}_{12} \\
& \widehat{R}_{22}
\end{array}
\right)
\right\|_2 \\
&\le &
\sigma_{k+1}\left(A\right) \sqrt[4]{1 + 2 \left(\frac{\left\|R_{22}\right\|_2}{\sigma_{k+1}\left(A\right)}\right)^4}.
\end{eqnarray*}

To obtain an upper bound of $\left\|R_{22}\right\|_2$ in \eqref{Eqn: the auxiliary lower bound of sigma_j(Sigma_k)} and
\eqref{Eqn: the auxiliary upper bound of approximate SVD}, we follow the analysis of SRQR in \cite{xiao2017fast}. 

From the analysis of \cite[Section IV]{xiao2017fast}, Algorithm \ref{Alg: SRQR} ensures that $g_1 \le \sqrt{\frac{1+\varepsilon}{1-\varepsilon}}$ and $g_2 \le g$, where $g_1$ and $g_2$ are defined by \eqref{Eqn:Rtilde2}. Here $0<\varepsilon<1$ is a user defined parameter to adjust the choice of oversampling size $p$ used in the TRQRCP initialization part in SRQR. $g > 1$ is a user defined parameter in the extra swapping part in SRQR. Let 
\begin{equation}\label{Eqn: definition of tau and widehat(tau)}
\tau \stackrel{def}{=} g_1 g_2 \frac{\|R_{22}\|_2}{\|R_{22}\|_{1,2}} \frac{\left\|\widetilde{R}^{-T}\right\|_{1,2}^{-1}}{\sigma_{l+1}\left(A\right)} \qquad \mbox{and} \qquad \widehat{\tau} \stackrel{def}{=} g_1 g_2 \frac{\|R_{22}\|_2}{\|R_{22}\|_{1,2}} \frac{\left\|R_{11}^{-T}\right\|_{1,2}^{-1}}{\sigma_k\left(\Sigma_k\right)},
\end{equation}
where $\widetilde{R}$ is defined by \eqref{Eqn:Rtilde}. Since $\frac{1}{\sqrt{n}}\|X\|_{1,2} \le \|X\|_2 \le \sqrt{n}\|X\|_{1,2}$ and $\sigma_i\left(X_1\right) \le \sigma_i\left(X\right)~\left(1\leq i \le \min\left(s,t\right)\right)$ for any matrix $X \in \mathbb{R}^{m \times n}$ and submatrix $X_1 \in \mathbb{R}^{s \times t}$ of $X$,
\begin{equation*}
\tau = g_1 g_2 \frac{\|R_{22}\|_2}{\|R_{22}\|_{1,2}} \frac{\left\|\widetilde{R}^{-T}\right\|_{1,2}^{-1}}{\sigma_{l+1}\left(\widetilde{R}\right)}\frac{\sigma_{l+1}\left(\widetilde{R}\right)}{\sigma_{l+1}\left(A\right)}\le g_1 g_2 \sqrt{\left(l+1\right)\left(n-l\right)}.
\end{equation*}
Using the fact that $\sigma_l\left(\begin{pmatrix}
R_{11} & R_{12}
\end{pmatrix}\right) = \sigma_l\left(\widehat{R}_{11}\right)$ and $\sigma_k\left(\Sigma_k\right) = \sigma_k\left(\begin{pmatrix}
\widehat{R}_{11} & \widehat{R}_{12}
\end{pmatrix}\right)$ by \eqref{Eqn: partial QRNP on R^T} and \eqref{Eqn: rank-k approximate svd of A by A (Pi widehat(Q1)) (Pi widehat(Q1))^T},
\begin{equation*}
\widehat{\tau} = g_1 g_2 \frac{\|R_{22}\|_2}{\|R_{22}\|_{1,2}} \frac{\left\|R_{11}^{-T}\right\|_{1,2}^{-1}}{\sigma_l\left(R_{11}\right)}\frac{\sigma_l\left(R_{11}\right)}{\sigma_l\left(\begin{pmatrix}
R_{11} & R_{12}
\end{pmatrix}\right)}\frac{\sigma_l\left(\begin{pmatrix}
R_{11} & R_{12}
\end{pmatrix}\right)}{\sigma_k\left(\Sigma_k\right)} \le g_1 g_2 \sqrt{l\left(n-l\right)},
\end{equation*}

By definition of $\tau$,
\begin{equation}\label{Eqn: expression of 2-norm of R_22}
\left\|R_{22}\right\|_2 = \tau \, \sigma_{l+1}\left(A\right).
\end{equation}
Plugging this into \eqref{Eqn: the auxiliary upper bound of approximate SVD} yields \eqref{Eqn: the upper bound of approximate SVD}.

By definition of $\widehat{\tau}$, we observe that
\begin{equation}\label{Eqn: the upper bound of 2-norm of R22}
\|R_{22}\|_2 \le \widehat{\tau} \sigma_k\left(\Sigma_k\right).
\end{equation}
By \eqref{Eqn: the auxiliary lower bound of sigma_j(Sigma_k)} and \eqref{Eqn: the upper bound of 2-norm of R22},
\begin{equation}\label{Eqn: equation one for result}
\sigma_j \left(\Sigma_k \right) \ge \frac{\sigma_j\left(A\right)}{\sqrt[4]{1+2\left(\frac{\left\|R_{22}\right\|_2}{\sigma_j \left(\Sigma_k \right)}\right)^4}} \ge \frac{\sigma_j\left(A\right)}{\sqrt[4]{1+2\left(\frac{\left\|R_{22}\right\|_2}{\sigma_k \left(\Sigma_k \right)}\right)^4}} \ge \frac{\sigma_j\left(A\right)}{\sqrt[4]{1+2\widehat{\tau}^4}}, ~ \left(1\le j \le k\right).
\end{equation}

\noindent On the other hand, using \eqref{Eqn: the auxiliary lower bound of sigma_j(Sigma_k)},
\begin{align*}
\sigma_j^4\left(A\right) &\le \sigma_j^4\left(\Sigma_k\right)\left(1+2\frac{\sigma_j^4\left(A\right)}{\sigma_j^4\left(\Sigma_k\right)}\frac{\left\|R_{22}\right\|_2^4}{\sigma_j^4\left(A\right)}\right) \\
& \le \sigma_j^4\left(\Sigma_k\right)\left(1+2\frac{\left(\sigma_j^4\left(\Sigma_k\right)+2\left\|R_{22}\right\|_2^4\right)}{\sigma_j^4\left(\Sigma_k\right)}\frac{\left\|R_{22}\right\|_2^4}{\sigma_j^4\left(A\right)}\right) \\
& \le \sigma_j^4\left(\Sigma_k\right)\left(1+2\left(1+2\frac{\left\|R_{22}\right\|_2^4}{\sigma_k^4\left(\Sigma_k\right)}\right)\frac{\left\|R_{22}\right\|_2^4}{\sigma_j^4\left(A\right)}\right),
\end{align*}
that is,
\begin{equation*}
\sigma_j\left(\Sigma_k\right) \ge \frac{\sigma_j\left(A\right)}{\sqrt[4]{1+\left(2+4\frac{\left\|R_{22}\right\|_2^4}{\sigma_k^4\left(\Sigma_k\right)}\right)\frac{\left\|R_{22}\right\|_2^4}{\sigma_j^4\left(A\right)}}}
\end{equation*}
Plugging \eqref{Eqn: expression of 2-norm of R_22} and \eqref{Eqn: the upper bound of 2-norm of R22} into this above equation, 
\begin{equation}\label{Eqn: equation two for result}
\sigma_j\left(\Sigma_k\right) \ge \frac{\sigma_j\left(A\right)}{\sqrt[4]{1 + \tau^4 \left(2+4\widehat{\tau}^4\right) \left(\frac{\sigma_{l+1}\left(A\right)}{\sigma_j\left(A\right)}\right)^4}}, \qquad \left(1\le j \le k\right).
\end{equation}
Combing \eqref{Eqn: equation one for result} and \eqref{Eqn: equation two for result}, we arrive at \eqref{Eqn: the upper bound of approximate SVD}.
\end{proof}

We note that \eqref{Eqn: the auxiliary lower bound of sigma_j(Sigma_k)} and \eqref{Eqn: the auxiliary upper bound of approximate SVD} still hold true if we replace $k$ by $l$.

Equation \eqref{Eqn: the lower bound of sigma_j(Sigma_k)} shows that under definitions \eqref{Eqn: definition of tau and widehat(tau)} of $\tau$ and $\widehat{\tau}$, Flip-Flop SRQR can reveal at least a dimension dependent fraction of all the leading singular values of $A$ and indeed approximate them very accurately in case they decay relatively quickly. Moreover, \eqref{Eqn: the upper bound of approximate SVD} shows that Flip-Flop SRQR can compute a rank-$k$ approximation that is up to a factor of $\sqrt[4]{1+2\tau^4 \left(\frac{\sigma_{l+1}\left(A\right)}{\sigma_{k+1}\left(A\right)}\right)^4}$ from optimal. In situations where singular values of $A$ decay relatively quickly, our rank-$k$ approximation is about as accurate as the truncated SVD with a choice of $l$ such that
\begin{equation*}
\frac{\sigma_{l+1}\left(A\right)}{\sigma_{k+1}\left(A\right)} = o\left(1\right).
\end{equation*}

\section{Numerical Experiments}\label{Sec:Num}
In this section, we demonstrate the effectiveness and efficiency of Flip-Flop SRQR (FFSRQR) algorithm in several numerical experiments. Firstly, we compare FFSRQR with other approximate SVD algorithms on matrix approximation. Secondly, we compare FFSRQR with other methods on tensor approximation problem using tensorlab toolbox \cite{Tensorlab2016}. Thirdly, we compare FFSRQR with other methods on the robust PCA problem and matrix completion problem. All experiments are implemented in Matlab R2016b on a MacBook Pro with a 2.9 GHz i5 processor and 8 GB memory. The underlying routines used in FFSRQR are written in Fortran. For a fair comparison, we turn off multi-threading functions in Matlab.

\subsection{Approximate Truncated SVD}
In this section, we compare FFSRQR with other four approximate SVD algorithms on low-rank matrices approximation. All tested methods are listed in Table \ref{Tab: Lists of comparison methods and details for approximate SVD tests}. The test matrices are:
\begin{itemize}
    \item[{\bf Type $1$:}] $A \in \mathbb{R}^{m \times n}$ \cite{stewart1999qlp} is defined by
    $A = U \, D \, V^T + 0.1 \, \sigma_{j}(D) E$
    where $U\in \mathbb{R}^{m \times s}, V \in \mathbb{R}^{n \times s}$ are column-orthonormal matrices, and $D \in \mathbb{R}^{s \times s}$ is a diagonal matrix with $s$ geometrically decreasing diagonal entries from $1$ to $10^{-3}$. $E \in \mathbb{R}^{m \times n}$ is a random matrix where the entries are independently sampled from a normal distribution $\mathcal{N}\left(0,1\right)$. In our numerical experiment, we test on three different random matrices. The square matrix has a size of $15000 \times 15000$; the short-fat matrix has a size of $1000 \times 15000$; the tall-skinny matrix has a size of $15000 \times 1000$. 
    \item[{\bf Type $2$:}] $A \in \mathbb{R}^{4929 \times 4929}$ is a real data matrix from the University of Florida sparse matrix collection \cite{davis2011university}. Its corresponding file name is HB/GEMAT11.
\end{itemize}

For a given matrix $A \in \mathbb{R}^{m \times n}$, the relative SVD approximation error is measured by ${\left\|A - U_k \Sigma_k V_k^T \right\|_F}/{\left\|A\right\|_F}$
where $\Sigma_k$ contains approximate top $k$ singular values, and $U_k,~V_k$ are corresponding approximate top $k$ singular vectors. The parameters used in FFSRQR, RSISVD, and LTSVD are listed in Table \ref{Tab: parameters for approximate svd method}.

\begin{table}
\begin{center}
\begin{tabular}{ll}
\hline
Method & Description \\
\hline          
LANSVD  &  Approximate SVD using Lanczos bidiagonalization with partial  \\
& reorthogonalization \cite{larsen1998lanczos}. We use its Matlab implementation in \\
& PROPACK \cite{Propack1998}.   \\
FFSRQR  &  Flip-Flop Spectrum-revealing QR factorization. \\
& We write its implementation using Mex functions that wrapped \\ & BLAS and LAPACK routines \cite{anderson1999lapack}. \\
RSISVD  &  Approximate SVD with randomized subspace iteration \cite{halko2011finding,gu2015subspace}. \\
& We use its Matlab implementation in tensorlab toolbox \cite{Tensorlab2016}. \\
LTSVD  &  Linear Time SVD \cite{drineas2006fast}. \\
& We use its Matlab implementation by Ma et al. \cite{ma2011fixed}. \\
\hline 
\end{tabular}
\end{center}
\caption{Methods for approximate SVD.}\label{Tab: Lists of comparison methods and details for approximate SVD tests}
\end{table}

\begin{table}
\begin{center}
\begin{tabular}{ll}
\hline
Method & Parameter \\
\hline   
FFSRQR & oversampling size $p=5$, integer $l = k$ \\
RSISVD & oversampling size $p=5$, subspace iteration $q = 1$ \\
LTSVD & probabilities $p_i = 1/n$ \\
\hline
\end{tabular}
\end{center}
\caption{Parameters used in RSISVD, FFSRQR, and LTSVD.}\label{Tab: parameters for approximate svd method}
\end{table}

Figure \ref{Fig: run time of matrix} through Figure \ref{Fig: top 20 singular values of matrix} show run time, relative approximation error, and top $20$ singular values comparison respectively on four different matrices. While LTSVD is faster than the other methods in most cases, the approximation error of LTSVD is significantly larger than all the other methods. In terms of accuracy, FFSRQR is comparable to LANSVD and RSISVD. In terms of speed, FFSRQR is faster than LANSVD. When target rank $k$ is small, FFSRQR is comparable to RSISVD, but FFSRQR is better when $k$ is larger. 

\begin{figure}
\begin{center}
\subfloat[Type 1: Random square matrix]{\label{Fig:random_sq_time}\includegraphics[width=0.5\linewidth]{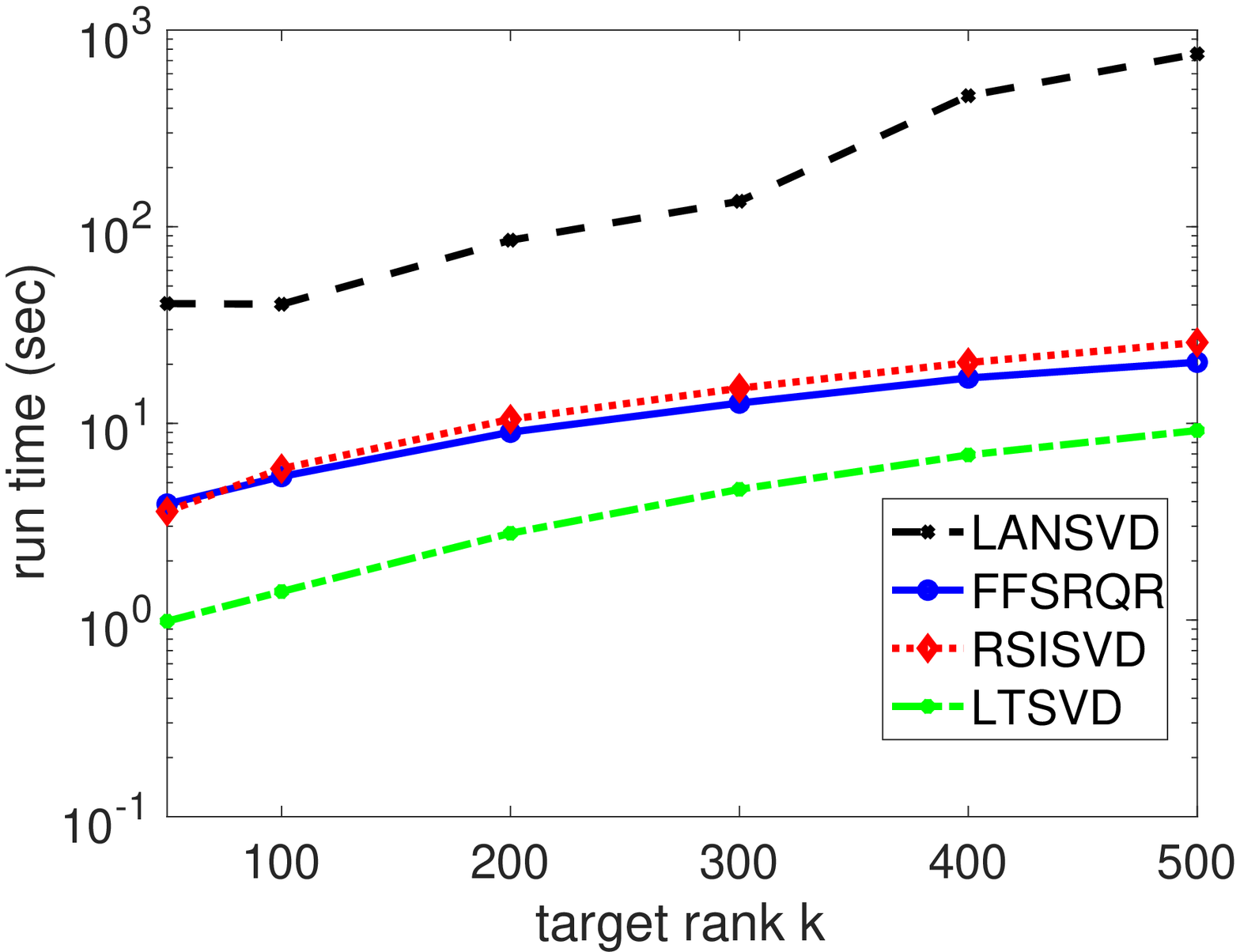}}
\subfloat[Type 1: Random short-fat matrix]{\label{Fig:random_sf_time}\includegraphics[width=0.5\linewidth]{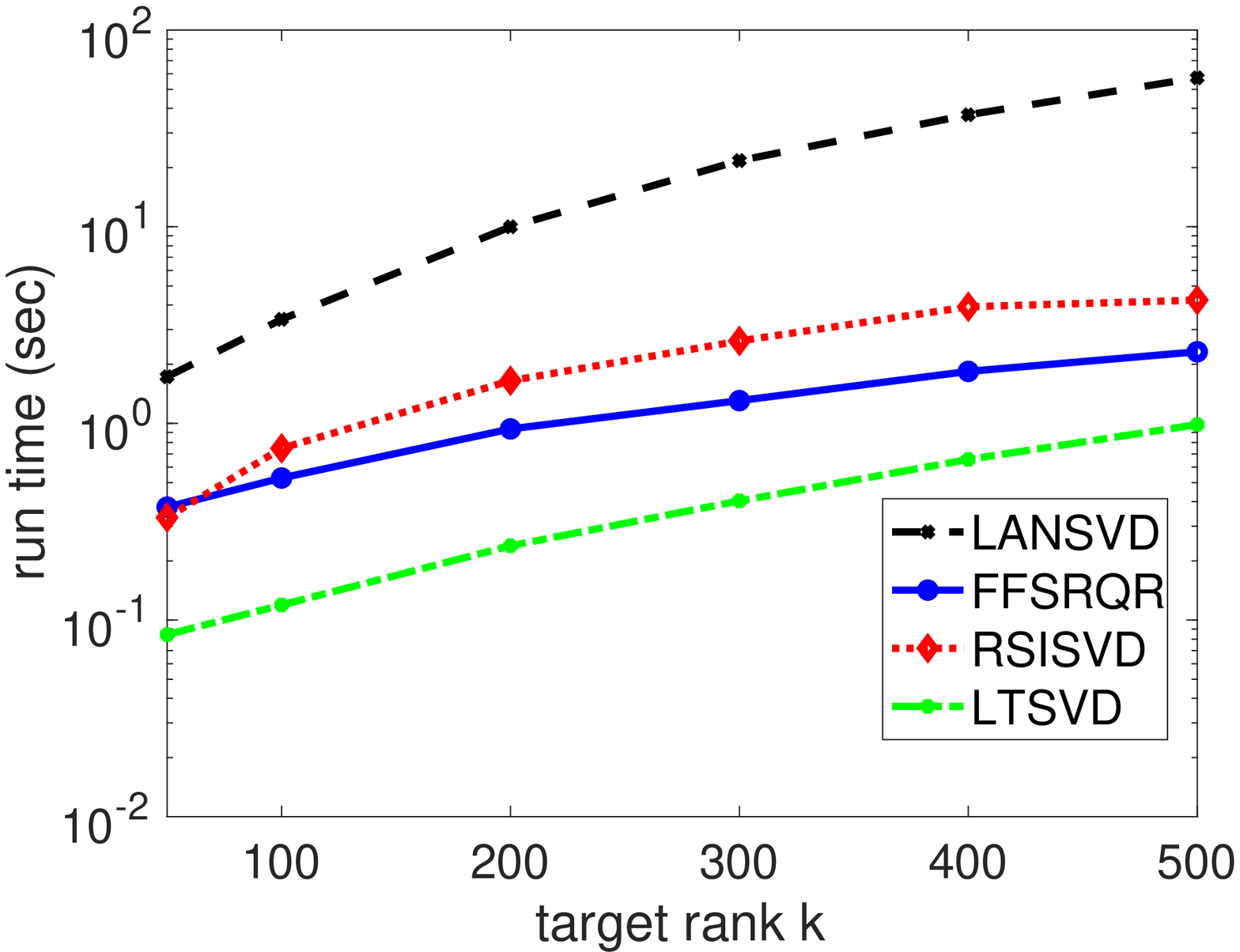}}\\
\subfloat[Type 1: Random tall-skinny matrix]{\label{Fig:random_ts_time}\includegraphics[width=0.5\linewidth]{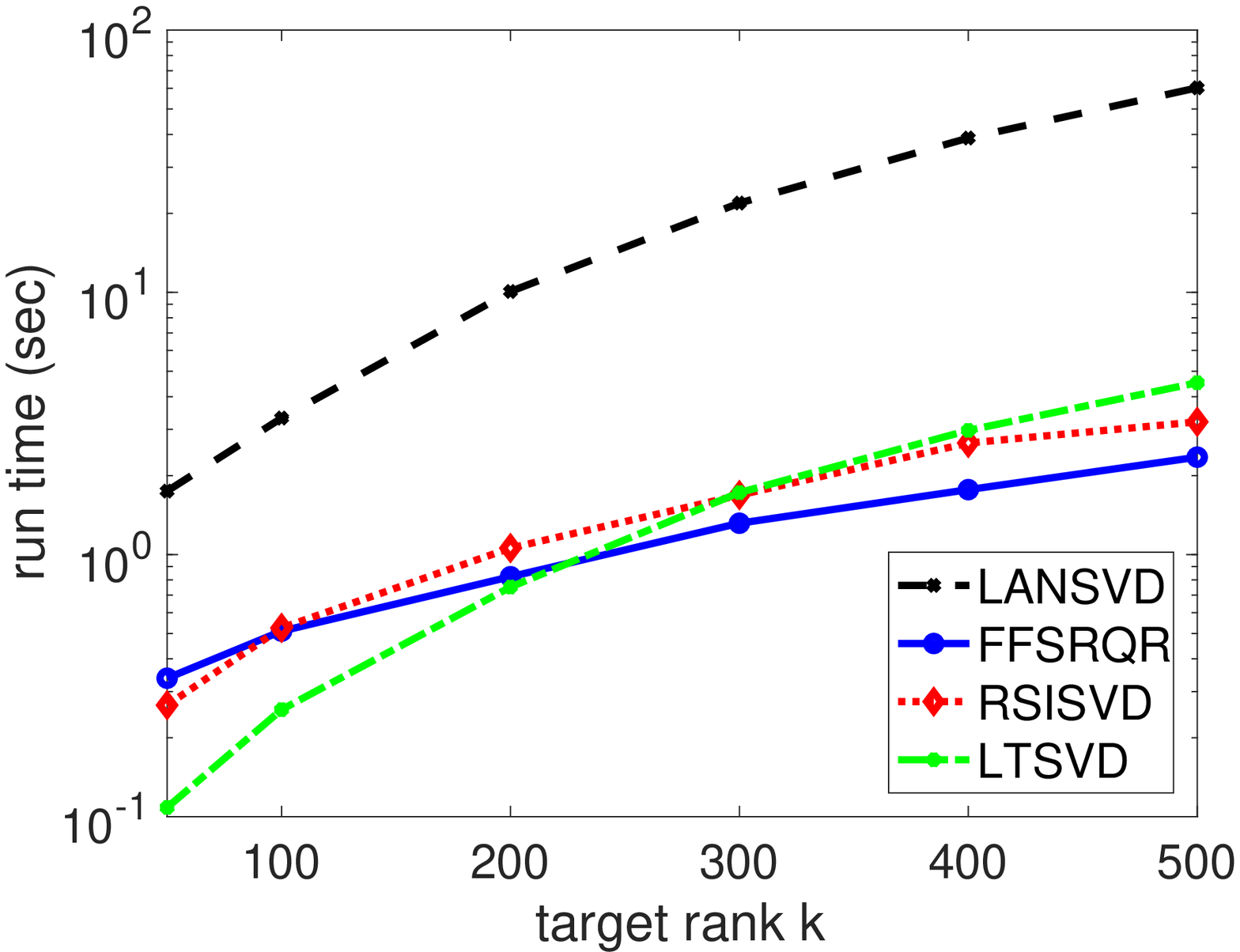}}
\subfloat[Type 2: GEMAT11]{\label{Fig:gemat11_time}\includegraphics[width=0.5\linewidth]{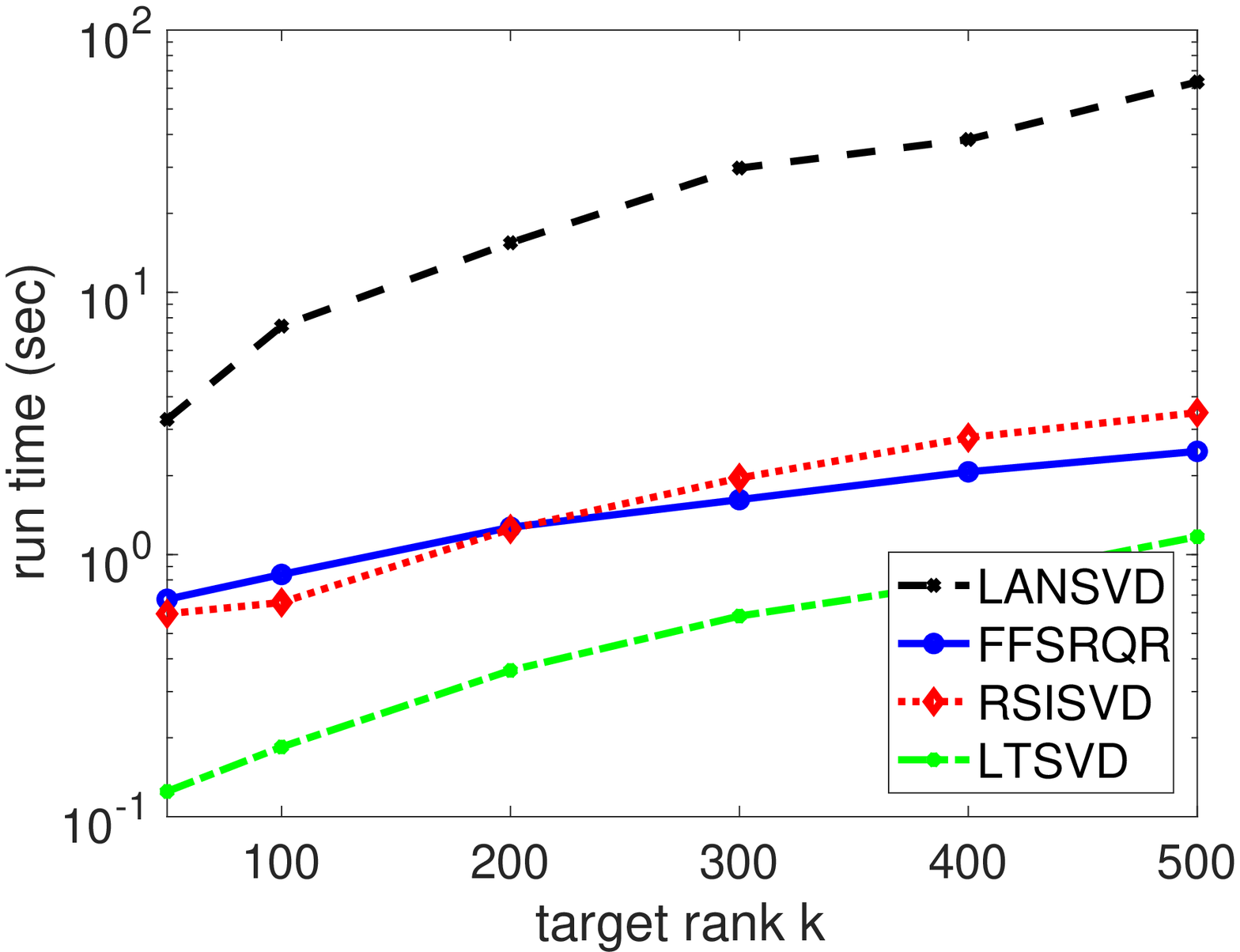}}
\end{center}
\caption{Run time comparison for approximate SVD algorithms.}\label{Fig: run time of matrix}
\end{figure}

\begin{figure}
\begin{center}
\subfloat[Type 1: Random square matrix]{\label{Fig:random_sq_error}\includegraphics[width=0.5\linewidth]{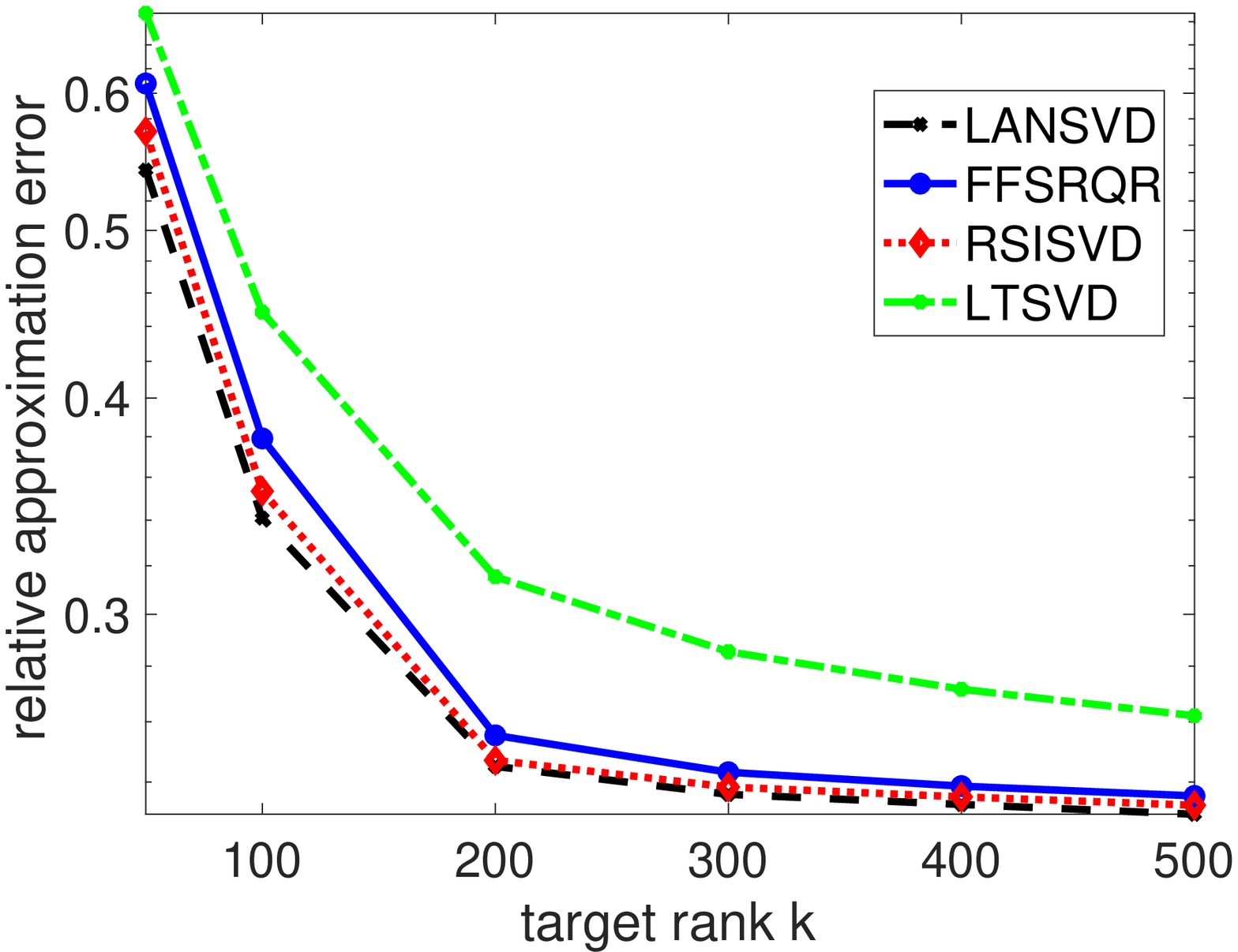}}
\subfloat[Type 1: Random short-fat matrix]{\label{Fig:random_sf_error}\includegraphics[width=0.5\linewidth]{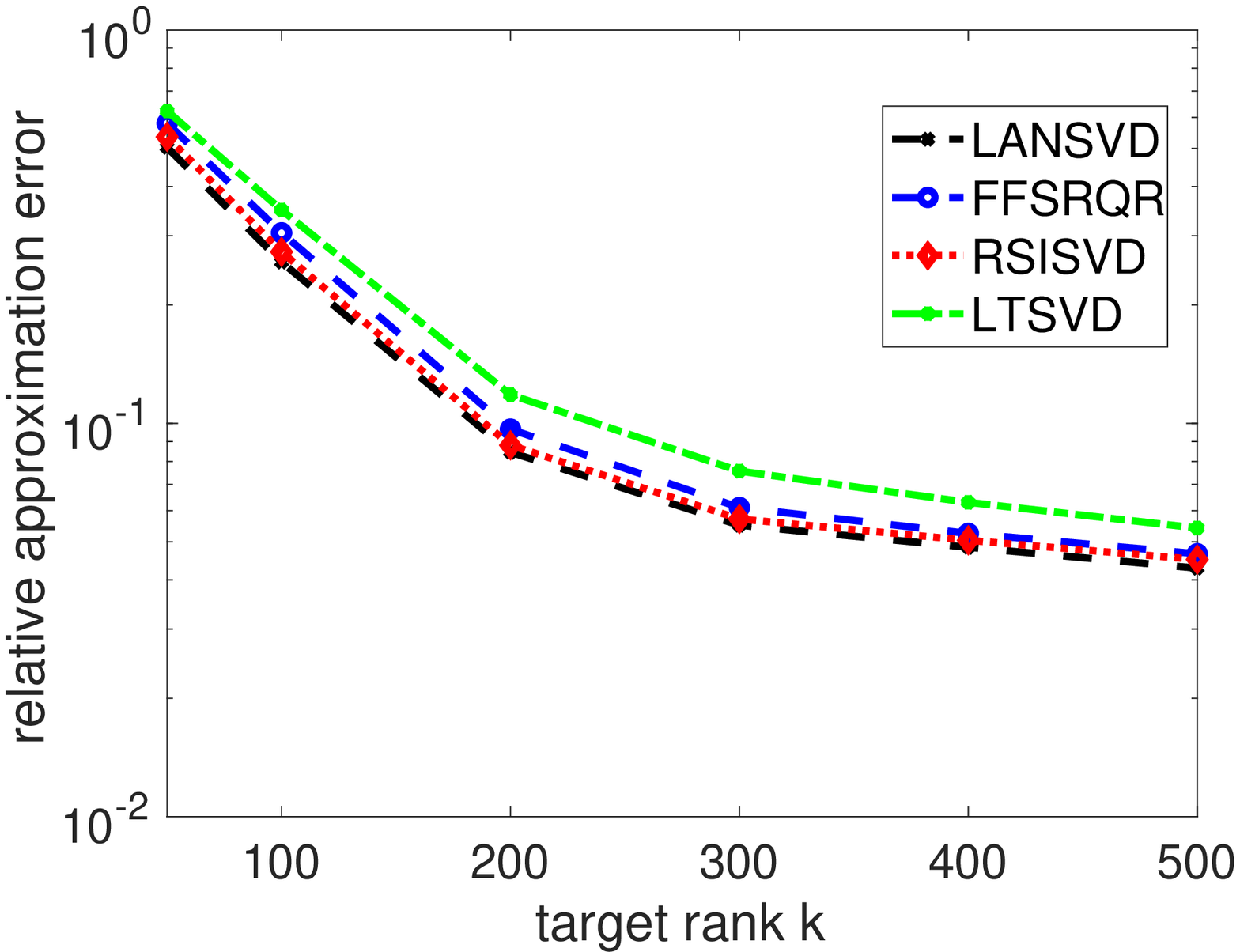}}\\
\subfloat[Type 1: Random tall-skinny matrix]{\label{Fig:random_ts_error}\includegraphics[width=0.5\linewidth]{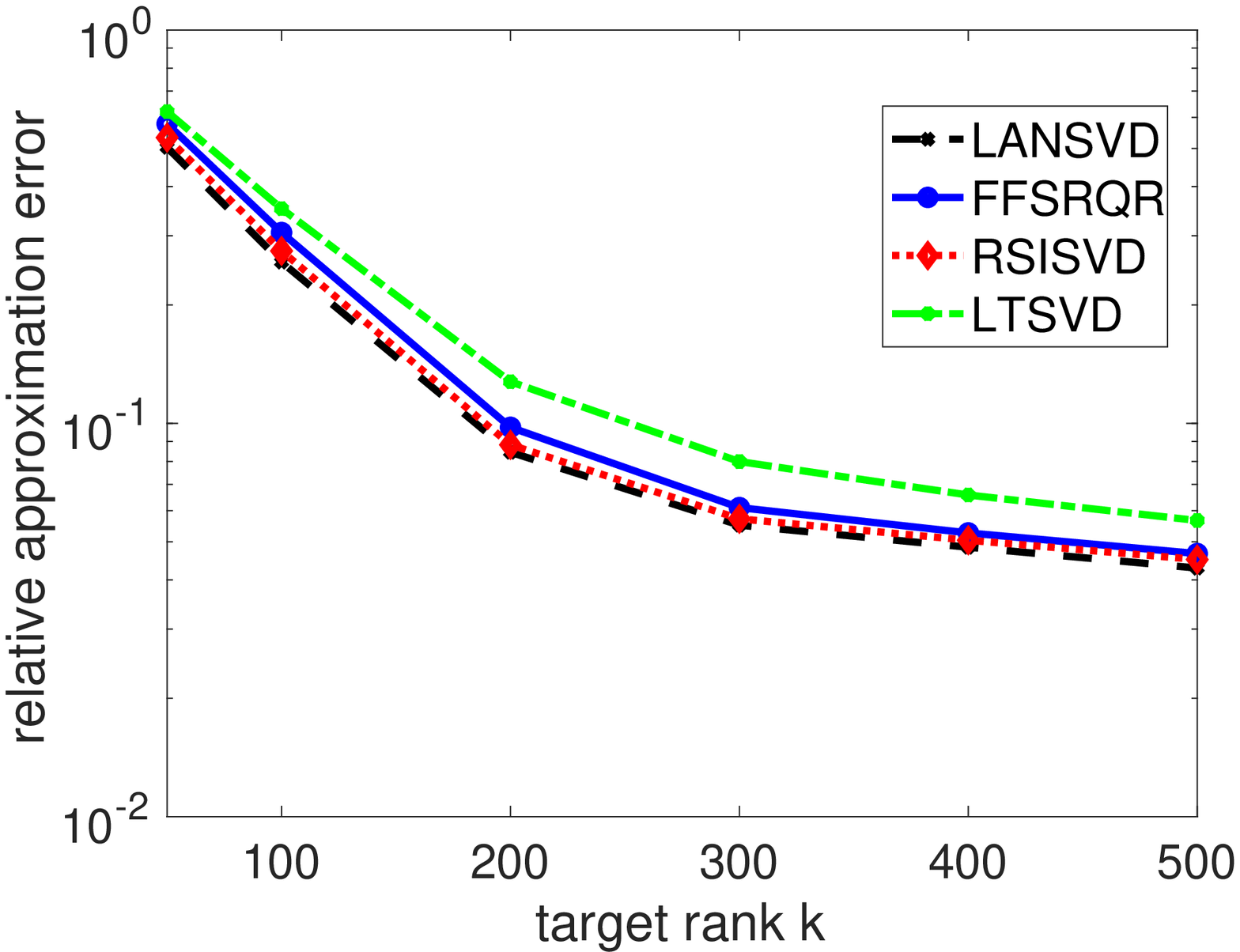}}
\subfloat[Type 2: GEMAT11]{\label{Fig:gemat11_error}\includegraphics[width=0.5\linewidth]{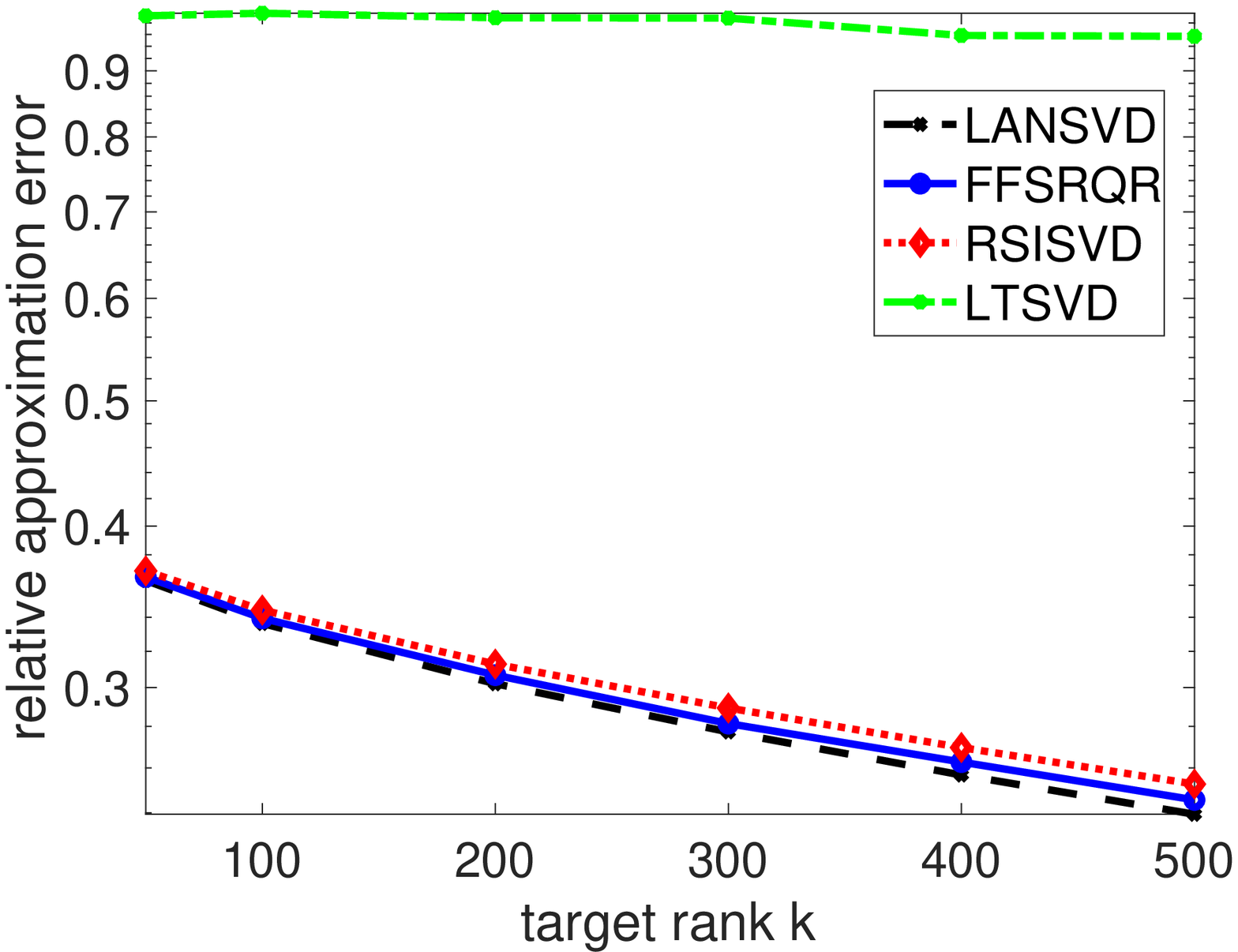}}
\end{center}
\caption{Relative approximation error comparison for approximate SVD algorithms.}\label{Fig: relative approximation error of matrix}
\end{figure}

\begin{figure}
\begin{center}
\subfloat[Type 1: Random square matrix]{\label{Fig:random_sq_sv}\includegraphics[width=0.5\linewidth]{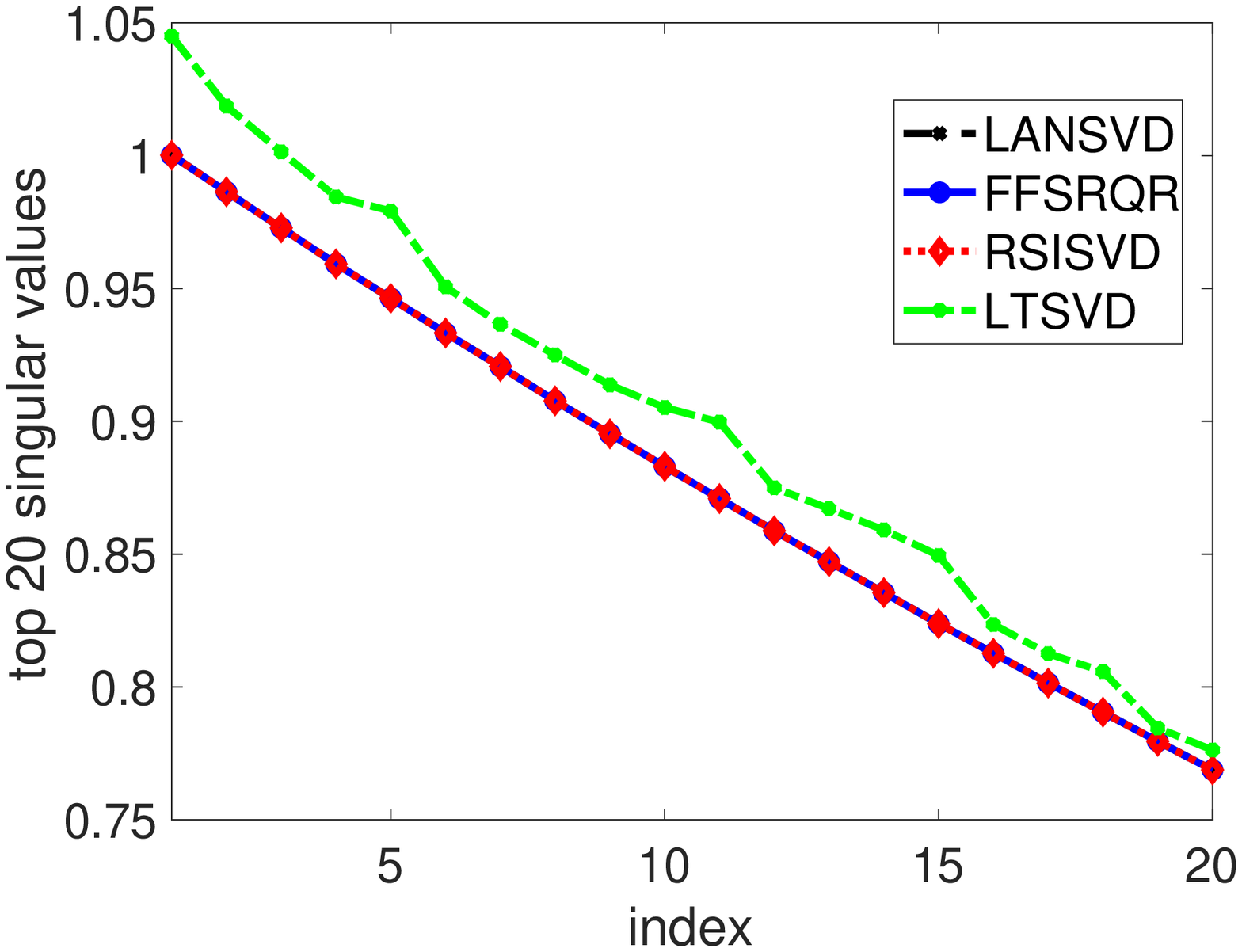}}
\subfloat[Type 1: Random short-fat matrix]{\label{Fig:random_sf_sv}\includegraphics[width=0.5\linewidth]{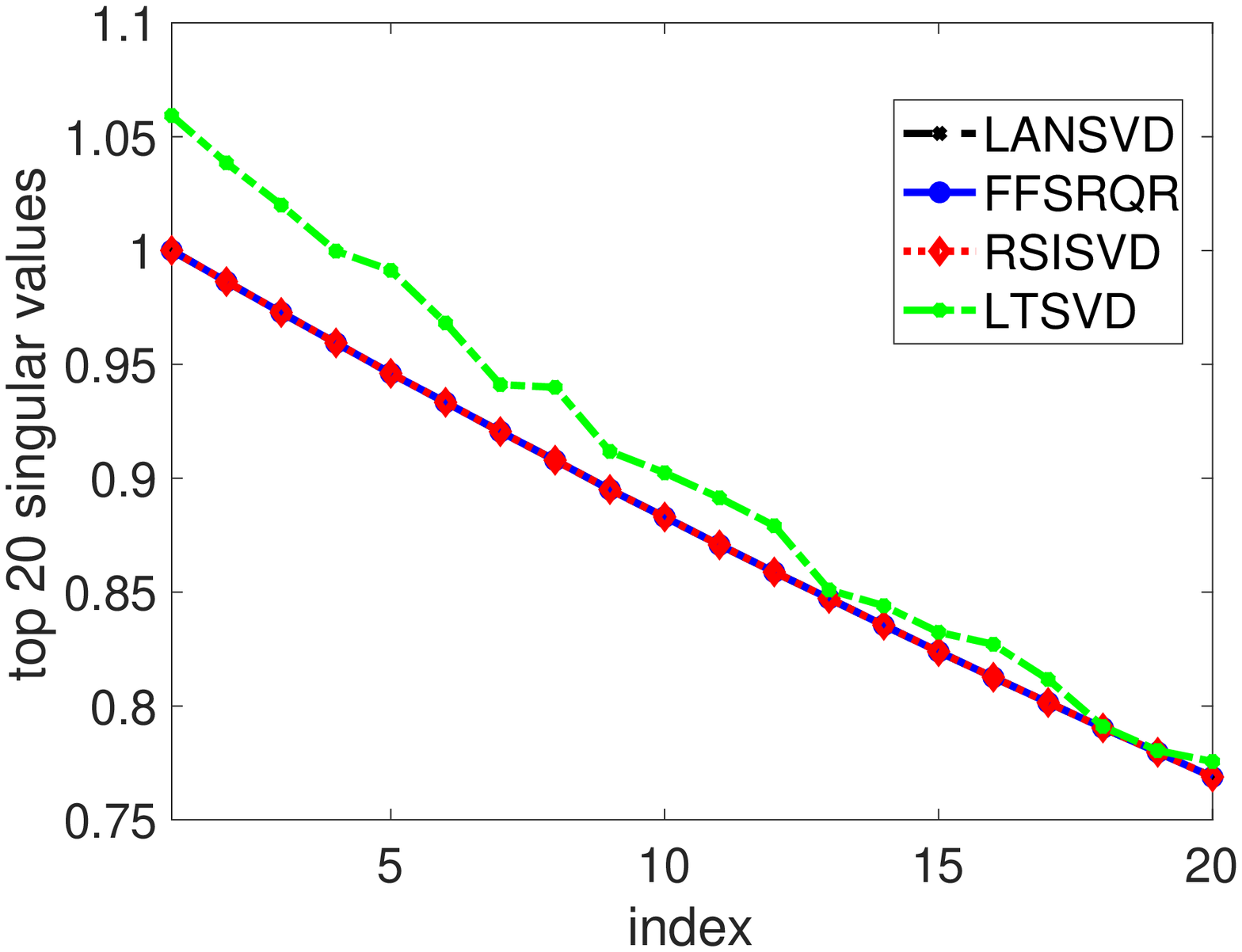}}\\
\subfloat[Type 1: Random tall-skinny matrix]{\label{Fig:random_ts_sv}\includegraphics[width=0.5\linewidth]{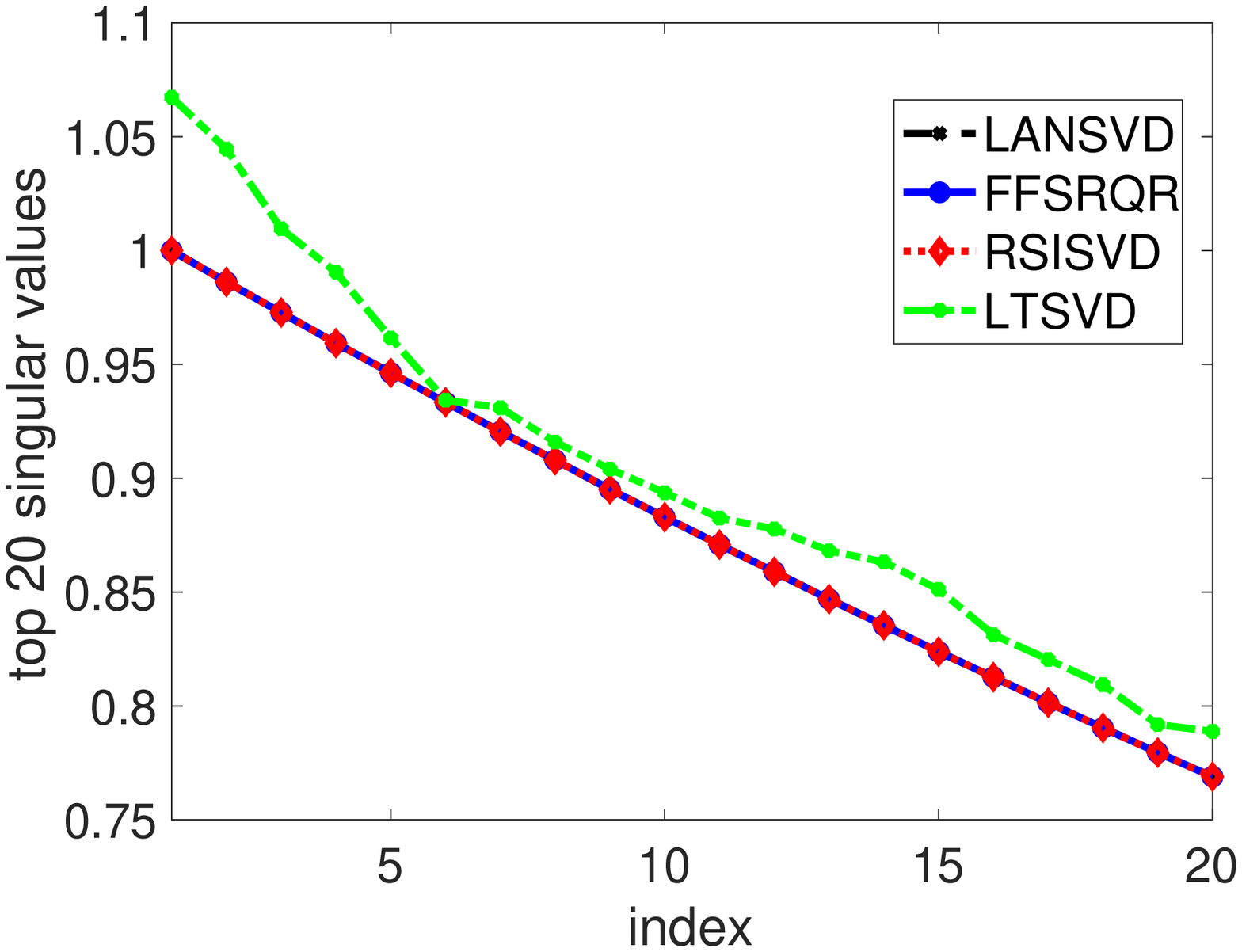}}
\subfloat[Type 2: GEMAT11]{\label{Fig:gemat11_sv}\includegraphics[width=0.5\linewidth]{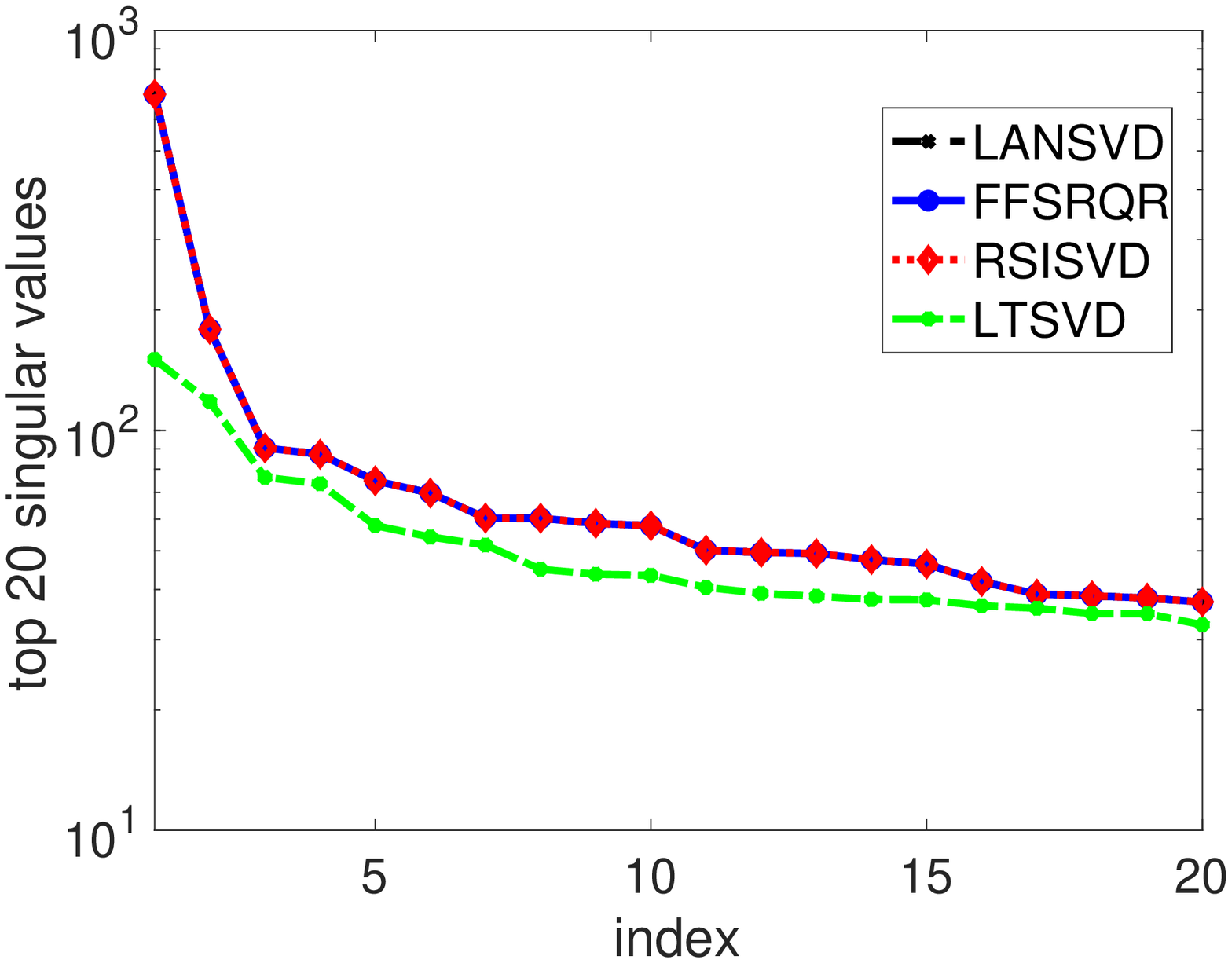}}
\end{center}
\caption{Top 20 singular values comparison for approximate SVD algorithms.}\label{Fig: top 20 singular values of matrix}
\end{figure}

\subsection{Tensor Approximation}
This section illustrates the effectiveness and efficiency of FFSRQR for computing approximate tensor. Sequentially truncated higher-order SVD (ST-HOSVD) \cite{andersson1998improving,vannieuwenhoven2012new} is one of the most efficient algorithms to compute Tucker decomposition of tensors, and the most costing part of this algorithm is to compute SVD or approximate SVD of the tensor unfoldings. Truncated SVD and randomized SVD with subspace iteration (RSISVD) are used in routines MLSVD and MLSVD\_RSI respectively in Matlab tensorlab toolbox \cite{Tensorlab2016}. Based on this Matlab toolbox, we implement ST-HOSVD using FFSRQR or LTSVD to do the SVD approximation. We name these two new routines by MLSVD\_FFSRQR and MLSVD\_LTSVD respectively. We compare these four routines in this numerical experiment. We also have Python codes for this tensor approximation numerical experiment. We don't list the results of python here but they are similar to those of Matlab. 
\subsubsection{A Sparse Tensor Example}
We test on a sparse tensor $\mathcal{X} \in \mathbb{R}^{n \times n \times n}$ of the following format \cite{sorensen2016deim, saibaba2016hoid},
\begin{equation*}
\mathcal{X} = \sum_{j=1}^{10} \frac{1000}{j}x_j \circ y_j \circ z_j + \sum_{j=11}^{n} \frac{1}{j} x_j \circ y_j \circ z_j,
\end{equation*}
where $x_j, y_j, z_j \in \mathbb{R}^n$ are sparse vectors with nonnegative entries. The symbol ``$\circ$'' represents the vector outer product. We compute a rank-$\left(k,k,k\right)$ Tucker decomposition $[\mathcal{G};U_1,U_2,U_3]$ using MLSVD, MLSVD\_FFSRQR,\ MLSVD\_RSI, and MLSVD\_LTSVD respectively. The relative approximation error is measured by ${\|\mathcal{X} - \mathcal{X}_k \|_F}/{\|\mathcal{X}\|_F}$ where $\mathcal{X}_k = \mathcal{G} \times_1 U_1 \times_2 U_2 \times_3 U_3$.

Figure \ref{Fig: results of sparse tensor with size 300x300x300} compares efficiency and accuracy of different methods on a $400 \times 400 \times 400$ sparse tensor approximation problem. MLSVD\_LTSVD is the fastest but the least accurate one. The other three methods have similar accuracy while MLSVD\_FFSRQR is faster when target rank $k$ is larger.

\begin{figure}
\begin{center}
\subfloat{\label{Fig:sparse_err}\includegraphics[width=0.5\linewidth]{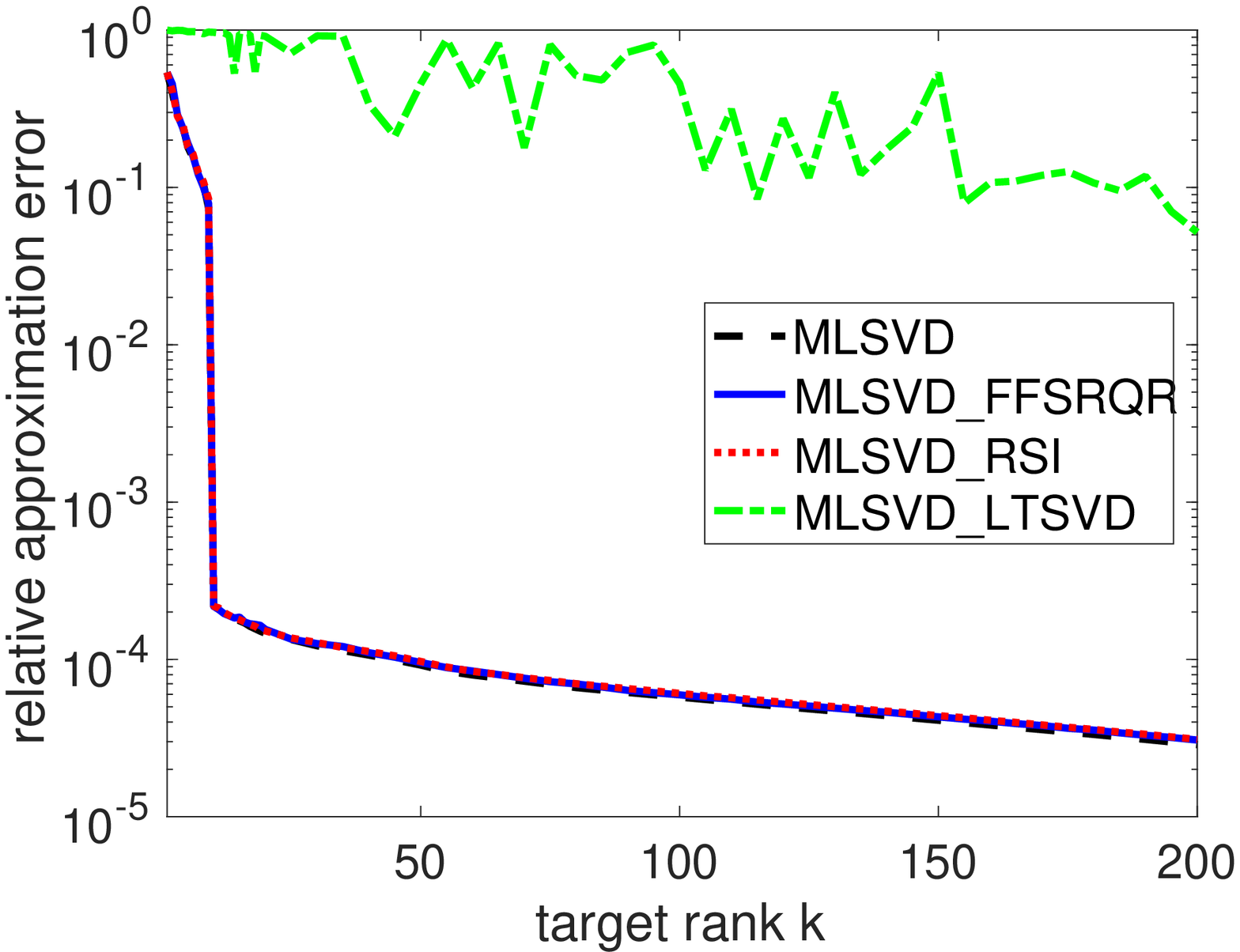}}
\subfloat{\label{Fig:sparse_time}\includegraphics[width=0.5\linewidth]{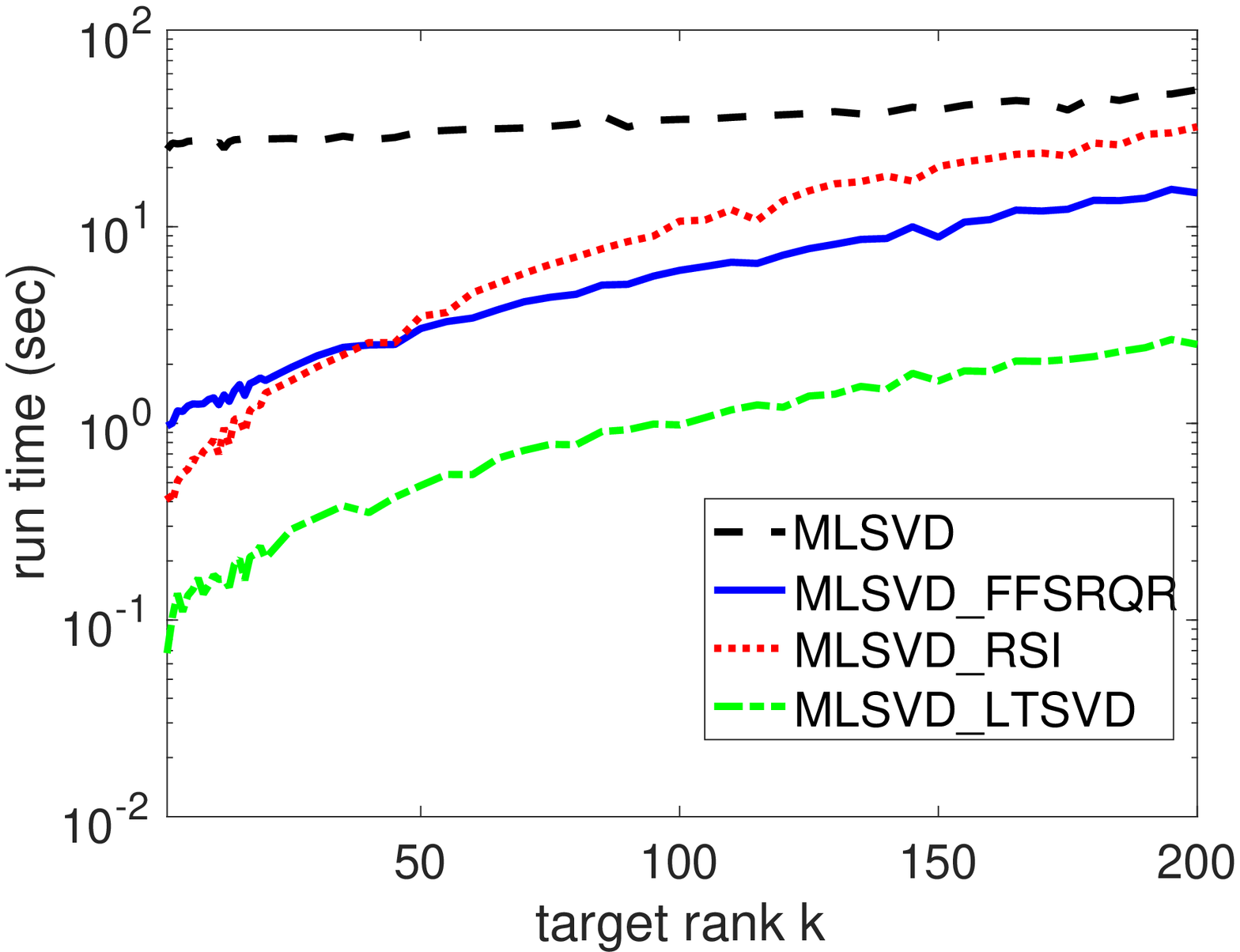}}
\end{center}
\caption{Run time and relative approximation error comparison on a sparse tensor.}\label{Fig: results of sparse tensor with size 300x300x300} 
\end{figure}

\subsubsection{Handwritten Digits Classification}
MNIST is a handwritten digits image data set created by Yann LeCun \cite{lecun1998gradient}. Every digit is represented by a $28 \times 28$ pixel image. Handwritten digits classification is to train a classification model to classify new unlabeled images. A HOSVD algorithm is proposed by Savas and Eld$\acute{e}$n \cite{savas2007handwritten} to classify handwritten digits. To reduce the training time, a more efficient ST-HOSVD algorithm is introduced in \cite{vannieuwenhoven2012new}.

We do handwritten digits classification using MNIST which consists of $60,000$ training images and $10,000$ test images. The number of training images in each class is restricted to $5421$ so that the training set are equally distributed over all classes. The training set is represented by a tensor $\mathcal{X}$ of size $786 \times 5421 \times 10$. The classification relies on Algorithm $2$ in \cite{savas2007handwritten}. We use various algorithms to obtain an approximation $\mathcal{X} \approx \mathcal{G} \times_1 U_1 \times_2 U_2 \times_3 U_3$ where the core tensor $\mathcal{G}$ has size $65 \times 142 \times 10$.

The results are summarized in Table \ref{Tab: handwritten classification results}. In terms of run time, our method MLSVD\_FFSRQR is comparable to MLSVD\_RSI while MLSVD is the most expensive one and MLSVD\_LTSVD is the fastest one. In terms of classification quality, MLSVD, MLSVD\_FFSRQR, and MLSVD\_RSI are comparable while MLSVD\_LTSVD is the least accurate one. 

\begin{table}
\begin{center}
\scalebox{0.85}{
\begin{tabular}{rcccc}
&  MLSVD & MLSVD\_FFSRQR & MLSVD\_RSI & MLSVD\_LTSVD \\
\hline
Training Time [sec] & $27.2121$ & $1.5455$ & $1.9343$ & $0.4266$ \\
Relative Model Error &  $0.4099$ & $0.4273$ &  $0.4247$ & $0.5162$\\
Classification Accuracy &    95.19\%  &  94.98\%  &  95.05\% &  92.59\% \\
\hline      
\end{tabular}
}
\end{center}
\caption{Comparison on handwritten digits classification.}\label{Tab: handwritten classification results}
\end{table}

\subsection{Solving Nuclear Norm Minimization Problem}
To show the effectiveness of FFSRQR algorithm in nuclear norm minimization problems, we investigate two scenarios: robust PCA \eqref{Eqn: nuclear norm minimization of RPCA problem} and matrix completion \eqref{Eqn: nuclear norm minimization of matrix completion problem}. The test matrix used in robust PCA is introduced in \cite{lin2010augmented} and the test matrices used in matrix completion are two real data sets. We use IALM method \cite{lin2010augmented} to solve both problems and IALM's code can be downloaded from \href{http://perception.csl.illinois.edu/matrix-rank/sample_code.html}{IALM}.

\subsubsection{Robust PCA}
To solve the robust PCA problem, we replace the approximate SVD part in IALM method \cite{lin2010augmented} by various methods. We denote the actual solution to the robust PCA problem by a matrix pair $\left(X^*,E^*\right) \in \mathbb{R}^{m \times n} \times \mathbb{R}^{m \times n}$. Matrix $X^* = X_L X_R^T$ where $X_L \in \mathbb{R}^{m \times k},~ X_R \in \mathbb{R}^{n \times k}$ are random matrices where the entries are independently sampled from normal distribution. Sparse matrix $E^*$ is a random matrix where its non-zero entries are independently sampled from a uniform distribution over the interval $[-500,500]$. The input to the IALM algorithm has the form $M = X^*+E^*$ and the output is denoted by $\left(\widehat{X}, \widehat{E}\right)$. In this numerical experiment, we use the same parameter settings as the IALM code for robust PCA: rank $k$ is $0.1m$ and number of non-zero entries in $E$ is $0.05m^2$. We choose the trade-off parameter $\lambda = 1 / \sqrt{\max\left(m,n\right)}$ as suggested by Cand$\grave{e}$s et al. \cite{candes2011robust}. The solution quality is measured by the normalized root mean square error ${\|\widehat{X}-X^*\|_F}/{\|X^*\|_F}$.

Table \ref{Tab: results of IALM on RPCA} includes relative error, run time, the number of non-zero entries in $\widehat{E}$ ($\|\widehat{E}\|_0$), iteration count, and the number of non-zero singular values (\#sv) in $\widehat{X}$ of IALM algorithm using different approximate SVD methods. We observe that IALM\_FFSRQR is faster than all the other three methods, while its error is comparable to IALM\_LANSVD and IALM\_RSISVD. IALM\_LTSVD is relatively slow and not effective. 

\begin{table}
\begin{center}
\scalebox{0.9}{\begin{tabular}{|c|c|ccccc|}
\hline
Size & Method & Error & Time (sec) & $\|\widehat{E}\|_0$ & Iter & \#sv \\
\hline
\multirow{4}{*}{$1000\times 1000$} & IALM\_LANSVD & $3.33e-07$ & $5.79e+00$ & $50000$ & $22$ & $100$ \\
& IALM\_FFSRQR & $2.79e-07$ & $1.02e+00$ & $50000$ & $25$ & $100$ \\
& IALM\_RSISVD & $3.36e-07$ & $1.09e+00$ & $49999$ & $22$ & $100$ \\
& IALM\_LTSVD & $9.92e-02$ & $3.11e+00$ & $999715$ & $100$ & $100$ \\
\hline
\multirow{4}{*}{$2000\times 2000$} & IALM\_LANSVD &$2.61e-07$ & $5.91e+01$ & $199999$ & $22$ & $200$ \\
& IALM\_FFSRQR & $1.82e-07$ & $6.93e+00$ & $199998$ & $25$ & $200$ \\
& IALM\_RSISVD & $2.63e-07$ & $7.38e+00$ & $199996$ & $22$ & $200$\\
& IALM\_LTSVD & $8.42e-02$ & $2.20e+01$ & $3998937$ & $100$ & $200$ \\
\hline
\multirow{4}{*}{$4000\times 4000$} & IALM\_LANSVD &$1.38e-07$ & $4.65e+02$ & $799991$ & $23$ & $400$ \\
& IALM\_FFSRQR & $1.39e-07$ & $4.43e+01$ & $800006$ & $26$ & $400$ \\
& IALM\_RSISVD & $1.51e-07$ & $5.04e+01$ & $799990$ & $23$ & $400$ \\
& IALM\_LTSVD & $8.94e-02$ & $1.54e+02$ & $15996623$ & $100$ & $400$ \\
\hline
\multirow{4}{*}{$6000\times 6000$} & IALM\_LANSVD &$1.30e-07$ & $1.66e+03$ & $1799982$ & $23$ & $600$ \\
& IALM\_FFSRQR & $1.02e-07$ & $1.42e+02$ & $1799993$ & $26$ & $600$  \\
& IALM\_RSISVD & $1.44e-07$ & $1.62e+02$ & $1799985$ & $23$ & $600$ \\
& IALM\_LTSVD & $8.58e-02$ & $5.55e+02$ & $35992605$ & $100$ & $600$ \\
\hline
\end{tabular}}
\end{center}
\caption{Comparison on robust PCA.}\label{tab: results of IALM on RPCA}
\end{table}

\subsubsection{Matrix Completion}
We solve matrix completion problems on two real data sets used in \cite{toh2010accelerated}: the Jester joke data set \cite{goldberg2001eigentaste} and the MovieLens data set \cite{herlocker1999algorithmic}. The Jester joke data set consists of $4.1$ million ratings for $100$ jokes from $73,421$ users and can be downloaded from the website \href{http://goldberg.berkeley.edu/jester-data/}{Jester}. We test on the following data matrices:
\begin{itemize}
    \item {\bf jester-$1$}: Data from $24,983$ users who have rated $36$ or more jokes;
    \item {\bf jester-$2$}: Data from $23,500$ users who have rated $36$ or more jokes;
    \item {\bf jester-$3$}: Data from $24,938$ users who have rated between $15$ and $35$ jokes;
    \item {\bf jester-all}: The combination of {\bf jester-$1$}, {\bf jester-$2$}, and {\bf jester-$3$}.
\end{itemize}

The MovieLens data set can be downloaded from \href{https://grouplens.org/datasets/movielens/}{MovieLens}. We test on the following data matrices:
\begin{itemize}
    \item {\bf movie-$100$K}: $100,000$ ratings of $943$ users for $1682$ movies;
    \item {\bf movie-$1$M}: $1$ million ratings of $6040$ users for $3900$ movies;
    \item {\bf movie-latest-small}: $100,000$ ratings of $700$ users for $9000$ movies.
\end{itemize}

For each data set, we let $M$ be the original data matrix where $M_{ij}$ stands for the rating of joke (movie) $j$ by user $i$ and $\Gamma$ be the set of indices where $M_{ij}$ is known. The matrix completion algorithm quality is measured by the Normalized Mean Absolute Error (NMAE) which is defined by 
\begin{equation*}
\mbox{NMAE} \stackrel{def}{=} \frac{\frac{1}{\left|\Gamma\right|} \sum_{\left(i,j\right) \in \Gamma} |M_{ij} - X_{ij}|}{r_{\max} - r_{\min}},
\end{equation*}
where $X_{ij}$ is the prediction of the rating of joke (movie) $j$ given by user $i$, and $r_{\min}, r_{\max}$ are lower and upper bounds of the ratings respectively. For the Jester joke data sets we set $r_{\min} = -10$ and $r_{\max} = 10$. For the MovieLens data sets we set $r_{\min} = 1$ and $r_{\max} = 5$.

Since $\left|\Gamma\right|$ is large, we randomly select a subset $\Omega$ from $\Gamma$ and then use Algorithm \ref{Alg: ialm_mc} to solve the problem \eqref{Eqn: nuclear norm minimization of matrix completion problem}. We randomly select $10$ ratings for each user in the Jester joke data sets, while we randomly choose about $50\%$ of the ratings for each user in the MovieLens data sets. Table \ref{Tab: parameters in IALM algorithm} includes parameter settings in the algorithms. The maximum iteration number is $100$ in IALM, and all other parameters are the same as those used in \cite{lin2010augmented}. 

The numerical results are included in Table \ref{Tab: Numerical results on real data sets}. We observe that IALM\_FFSRQR achieves almost the same recoverability as other methods except for IALM\_LTSVD, and is slightly faster than IALM\_RSISVD for these two data sets.

\begin{table}
\begin{center}
\begin{tabular}{ccccc}
\hline
Data set & $m$ & $n$ & $|\Gamma|$ & $|\Omega|$  \\
\hline
jester-1 & $24983$ & $100$ & $1.81e+06$ & $249830$  \\
jester-2 & $23500$ & $100$ & $1.71e+06$ & $235000$  \\
jester-3 & $24938$ & $100$ & $6.17e+05$ & $249384$  \\
jester-all & $73421$ & $100$ & $4.14e+06$ & $734210$  \\
moive-100K & $943$ & $1682$ & $1.00e+05$ & $49918$   \\
moive-1M & $6040$ & $3706$ & $1.00e+06$ & $498742$   \\
moive-latest-small & $671$ & $9066$ & $1.00e+05$ & $52551$  \\
\hline
\end{tabular}
\end{center}
\caption{Parameters used in the IALM method on matrix completion.}\label{Tab: parameters in IALM algorithm}
\end{table}

\begin{table}
\begin{center}
\scalebox{0.85}{\begin{tabular}{clcccccc}
\hline
Data set & Method & Iter & Time & NMAE & \#sv & $\sigma_{\max}$ & $\sigma_{min}$ \\
\hline
\multirow{4}{*}{jester-1} & IALM-LANSVD & $12$ & $7.06e+00$ & $1.84e-01$ & $100$ & $2.14e+03$ & $1.00e+00$ \\
& IALM-FFSRQR & $12$ & $3.44e+00$ & $1.69e-01$ & $100$ & $2.28e+03$ & $1.00e+00$ \\
& IALM-RSISVD & $12$ & $3.75e+00$ & $1.89e-01$ & $100$ & $2.12e+03$ & $1.00e+00$ \\
& IALM-LTSVD & $100$ & $2.11e+01$ & $1.74e-01$ & $62$ & $3.00e+03$ & $1.00e+00$ \\
\hline
\multirow{4}{*}{jester-2} & IALM-LANSVD & $12$ & $6.80e+00$ & $1.85e-01$ & $100$ & $2.13e+03$ & $1.00e+00$ \\
& IALM-FFSRQR & $12$ & $2.79e+00$ & $1.70e-01$ & $100$ & $2.29e+03$ & $1.00e+00$  \\
& IALM-RSISVD & $12$ & $3.59e+00$ & $1.91e-01$ & $100$ & $2.12e+03$ & $1.00e+00$ \\
& IALM-LTSVD & $100$ & $2.03e+01$ & $1.75e-01$ & $58$ & $2.96e+03$ & $1.00e+00$  \\
\hline
\multirow{4}{*}{jester-3} & IALM-LANSVD & $12$ & $7.05e+00$ & $1.26e-01$ & $99$ & $1.79e+03$ & $1.00e+00$ \\
& IALM-FFSRQR &  $12$ & $3.03e+00$ & $1.22e-01$ & $100$ & $1.71e+03$ & $1.00e+00$ \\
& IALM-RSISVD & $12$ & $3.85e+00$ & $1.31e-01$ & $100$ & $1.78e+03$ & $1.00e+00$ \\
& IALM-LTSVD & $100$ & $2.12e+01$ & $1.33e-01$ & $55$ & $2.50e+03$ & $1.00e+00$ \\
\hline
\multirow{4}{*}{jester-all} & IALM-LANSVD & $12$ & $2.39e+01$ & $1.72e-01$ & $100$ & $3.56e+03$ & $1.00e+00$ \\
& IALM-FFSRQR & $12$ & $1.12e+01$ & $1.62e-01$ & $100$ & $3.63e+03$ & $1.00e+00$ \\
& IALM-RSISVD & $12$ & $1.34e+01$ & $1.82e-01$ & $100$ & $3.47e+03$ & $1.00e+00$ \\
& IALM-LTSVD & $100$ & $6.99e+01$ & $1.68e-01$ & $52$ & $4.92e+03$ & $1.00e+00$ \\
\hline
\multirow{4}{*}{moive-100K} & IALM-LANSVD & $29$ & $2.86e+01$ & $1.83e-01$ & $285$ & $1.21e+03$ & $1.00e+00$ \\
& IALM-FFSRQR & $30$ & $4.55e+00$ & $1.67e-01$ & $295$ & $1.53e+03$ & $1.00e+00$ \\
& IALM-RSISVD & $29$ & $4.82e+00$ & $1.82e-01$ & $285$ & $1.29e+03$ & $1.00e+00$ \\
& IALM-LTSVD & $48$ & $1.42e+01$ & $1.47e-01$ & $475$ & $1.91e+03$ & $1.00e+00$ \\
\hline
\multirow{4}{*}{moive-1M} & IALM-LANSVD & $50$ & $7.40e+02$ & $1.58e-01$ & $495$ & $4.99e+03$ & $1.00e+00$ \\
& IALM-FFSRQR & $53$ & $2.07e+02$ & $1.37e-01$ & $525$ & $6.63e+03$ & $1.00e+00$ \\
& IALM-RSISVD & $50$ & $2.23e+02$ & $1.57e-01$ & $495$ & $5.35e+03$ & $1.00e+00$ \\
& IALM-LTSVD & $100$ & $8.50e+02$ & $1.17e-01$ & $995$ & $8.97e+03$ & $1.00e+00$ \\
\hline
\multirow{4}{*}{\shortstack[c]{moive-latest\\-small}}& IALM-LANSVD & $31$ & $1.66e+02$ & $1.85e-01$ & $305$ & $1.13e+03$ & $1.00e+00$ \\
& IALM-FFSRQR & $31$ & $1.96e+01$ & $2.00e-01$ & $305$ & $1.42e+03$ & $1.00e+00$ \\
 & IALM-RSISVD & $31$ & $2.85e+01$ & $1.91e-01$ & $305$ & $1.20e+03$ & $1.00e+00$ \\
 & IALM-LTSVD & $63$ & $4.02e+01$ & $2.08e-01$ & $298$ & $1.79e+03$ & $1.00e+00$  \\
\hline
\end{tabular}}
\end{center}
\caption{Comparison on matrix completion.}\label{Tab: Numerical results on real data sets}
\end{table}

\section{Conclusions}
We presented the Flip-Flop SRQR factorization, a variant of QLP factorization, to compute low-rank matrix approximations. The Flip-Flop SRQR algorithm uses SRQR factorization to initialize the truncated version of column pivoted QR factorization and then form an LQ factorization. For the numerical results presented, the errors in the proposed algorithm were comparable to those obtained from the other state-of-the-art algorithms. This new algorithm is cheaper to compute and produces quality low-rank matrix approximations. Furthermore, we prove singular value lower bounds and residual error upper bounds for the Flip-Flop SRQR factorization. In situations where singular values of the input matrix decay relatively quickly, the low-rank approximation computed by SRQR is guaranteed to be as accurate as truncated SVD. We also perform complexity analysis to show that Flip-Flop SRQR is faster than approximate SVD with randomized subspace iteration. Future work includes reducing the overhead cost in Flip-Flop SRQR and implementing Flip-Flop SRQR algorithm on distributed memory machines for popular applications such as distributed PCA. 

\newpage 
\section{Appendix}
\subsection{Approximate SVD with Randomized Subspace Iteration}
Randomized subspace iteration was proposed in \cite[Algorithm 4.4]{halko2011finding} to compute an orthonormal matrix whose range approximates the range of $A$. An approximate SVD can be computed using the aforementioned orthonormal matrix \cite[Algorithm 5.1]{halko2011finding}.
Randomized subspace iteration is used in routine MLSVD\_RSI in Matlab toolbox tensorlab \cite{Tensorlab2016}, and MLSVD\_RSI is by far the most efficient function to compute ST-HOSVD we can find in Matlab. We summarize approximate SVD with randomized subspace iteration pseudocode in Algorithm \ref{Alg: randomized SVD with randomized subspace iteration}.

\begin{algorithm}
\caption{Approximate SVD with Randomized Subspace Iteration}\label{Alg: randomized SVD with randomized subspace iteration}
\begin{algorithmic}
\STATE $\textbf{Inputs:}$
\STATE Matrix $A \in \mathbb{R}^{m \times n}$. Target rank $k$. Oversampling size $p\ge 0$. Number of iterations $q \ge 1$.
\STATE $\textbf{Outputs:}$
\STATE $U\in \mathbb{R}^{m \times k}$ contains the approximate top $k$ left singular vectors of $A$.
\STATE $\Sigma \in \mathbb{R}^{k \times k}$ contains the approximate top $k$ singular values of $A$.
\STATE $V\in \mathbb{R}^{n \times k}$ contains the approximate top $k$ right singular vectors of $A$.
\STATE $\textbf{Algorithm:}$
\STATE Generate i.i.d Gaussian matrix $\Omega \in \mathcal{N}\left(0,1\right)^{n \times \left(k+p\right)}$.
\STATE Compute $B = A \Omega$.
\STATE $[Q,\sim] = qr\left(B,0\right)$
\FOR{$i=1:q$}
\STATE $B = A^T * Q$
\STATE $[Q,\sim] = qr\left(B,0\right)$
\STATE $B = A * Q$
\STATE $[Q,\sim] = qr\left(B,0\right)$
\ENDFOR
\STATE $B = Q^T * A$
\STATE $[U,\Sigma,V] = svd\left(B\right)$
\STATE $U = Q * U$
\STATE $U = U\left(:,1:k\right)$
\STATE $\Sigma = \Sigma\left(1:k,1:k\right)$
\STATE $V = V\left(:,1:k\right)$
\end{algorithmic}
\end{algorithm}

Now we perform a complexity analysis on approximate SVD with randomized subspace iteration. We first note that 
\begin{enumerate}
  \item The cost of generating a random matrix is negligible.

  \item The cost of computing $B = A \Omega$ is $2 m n \left(k+p\right)$.

  \item In each QR step $[Q,\sim] = qr\left(B,0\right)$,  the cost of computing the QR factorization of $B$ is $2 m \left(k+p\right)^2 - \frac{2}{3}\left(k+p\right)^3$ (c.f. \cite{trefethen1997numerical}), and the cost of forming the first $\left(k+p\right)$ columns in the full $Q$ matrix is $m \left(k+p\right)^2 + \frac{1}{3}\left(k+p\right)^3$.
\end{enumerate}

Now we count the flops for each $i$ in the ${\bf for}$ loop: 
\begin{enumerate}
    \item The cost of computing $B = A^T * Q$ is $2 m n \left(k+p\right)$;
    \item The cost of computing $[Q,\sim] = qr\left(B,0\right)$ is $2 n \left(k+p\right)^2 - \frac{2}{3}\left(k+p\right)^3$, and the cost of forming the first $\left(k+p\right)$ columns in the full $Q$ matrix is $n \left(k+p\right)^2 + \frac{1}{3}\left(k+p\right)^3$;
    \item The cost of computing $B = A * Q$ is $2 m n \left(k+p\right)$;
    \item The cost of computing $[Q,\sim] = qr\left(B,0\right)$ is $2 m \left(k+p\right)^2 - \frac{2}{3}\left(k+p\right)^3$, and the cost of forming the first $\left(k+p\right)$ columns in the full $Q$ matrix is $m \left(k+p\right)^2 + \frac{1}{3}\left(k+p\right)^3$.  
\end{enumerate}
Putting together, the cost of running the ${\bf for}$ loop $q$ times is 
\[ q\left(4 m n \left(k + p\right) + 3 \left(m + n\right) \left(k + p\right)^2 - \frac{2}{3}\left(k+p\right)^3\right).\]

Additionally, the cost of computing $B = Q^T * A$ is $2 m n \left(k+p\right)$; the cost of doing SVD of $B$ is $O(n \left(k+p\right)^2)$; and the cost of computing $U = Q * U$ is $2 m \left(k+p\right)^2$.

Now assume $k+p \ll \min\left(m,n\right)$ and omit the lower-order terms, then we arrive at $\left(4q+4\right)mn\left(k+p\right)$ as the complexity of approximate SVD with randomized subspace iteration. In practice, $q$ is usually chosen to be an integer between $0$ and $2$.

\bibliographystyle{siamplain}
\bibliography{references}
\end{document}